\documentclass[12pt,a4paper]{article}
\pdfoutput=1
\usepackage[english]{babel}
\RequirePackage{ucs}
\RequirePackage[utf8x]{inputenc}
\RequirePackage{autofe}
\RequirePackage{lmodern}

\RequirePackage{amsthm,amsmath}
\RequirePackage{amssymb,amstext}
\RequirePackage{amsfonts}
\RequirePackage{amsbsy} 
\RequirePackage{mathrsfs}
\RequirePackage{bm}
\RequirePackage{bbm}
\RequirePackage{multirow}

\usepackage[pdfpagemode=UseOutlines ,plainpages=false,
hyperindex=true,colorlinks=true]{hyperref}
	\makeatletter
	\Hy@breaklinkstrue
	\makeatother
\usepackage{color}
\definecolor{darkred}{rgb}{0.6,0.0,0.1}
\definecolor{darkgreen}{rgb}{0,0.5,0}
\definecolor{darkblue}{rgb}{0,0,0.5}
\hypersetup{colorlinks ,linkcolor=darkblue ,filecolor=darkgreen
,urlcolor=darkblue ,citecolor=black}

\RequirePackage{ifthen}
\RequirePackage{graphicx}
\RequirePackage{xcolor}
\RequirePackage{calc}
\RequirePackage{times}
\usepackage{url}
\usepackage{verbatim}

\RequirePackage{tikz}
\usepackage{ulem}%%% Texte à barrer
\RequirePackage{paralist}
\usepackage{enumitem}

\usepackage{natbib}
\usepackage{hypernat}
\renewcommand{\cite}{\citet}
\usepackage{./DP_NPE_dep-abbrev-package}

\definecolor{dgreen}{rgb}{0,0.5,0}
\definecolor{dblue}{rgb}{0,0,0.9}
\definecolor{dred}{rgb}{0.6,0.0,0.1}
\definecolor{dgold}{rgb}{0.5,0.3,0.0}
\definecolor{dvio}{rgb}{0.6,0.3,0.5}
\definecolor{gray}{rgb}{0.5,0.5,0.5}
\definecolor{dbraun}{rgb}{.5,0.2,0}
\definecolor{jaune}{RGB}{255,237,111}
\definecolor{tagada}{RGB}{237,162,153}
\definecolor{violett}{RGB}{250,50,100}
\definecolor{vert}{RGB}{190,242,182}

\newcommand{\dr}{\color{dred}}

\newtheoremstyle{mysc}% name
  {3pt}%      Space above
  {3pt}%      Space below
  {\it}%         Body font
  {}%         Indent amount (empty = no indent, \parindent = para indent)
  {\color{darkred}\sc}% Thm head font
  {.}%        Punctuation after thm head
  {.5em}%     Space after thm head: " " = normal interword space;
  {}%         Thm head spec (can be left empty, meaning `normal')

\newtheoremstyle{myas}% name
  {3pt}%      Space above
  {3pt}%      Space below
  {\it}%         Body font
  {}%         Indent amount (empty = no indent, \parindent = para indent)
  {\color{darkblue}\sc}% Thm head font
  {.}%        Punctuation after thm head
  {.5em}%     Space after thm head: " " = normal interword space;
  {}%         Thm head spec (can be left empty, meaning `normal')

\newtheoremstyle{myex}% name
  {10pt}%      Space above
  {10pt}%      Space below
  {\it}%         Body font
  {}%         Indent amount (empty = no indent, \parindent = para indent)
  {\color{darkred}\sc}% Thm head font
  {.}%        Punctuation after thm head
  {.5em}%     Space after thm head: " " = normal interword space;
  {}%         Thm head spec (can be left empty, meaning `normal')

\theoremstyle{mysc}\newtheorem{prop}{Proposition}[section]
\theoremstyle{mysc}\newtheorem{coro}[prop]{Corollary}
\theoremstyle{mysc}\newtheorem{theo}[prop]{Theorem}
\theoremstyle{mysc}\newtheorem{lem}[prop]{Lemma}
\theoremstyle{myas}

\theoremstyle{myex}\newtheorem{rem}{Remark}
\theoremstyle{myex}\newtheorem{illu}[rem]{Illustration}
 
\numberwithin{equation}{section} 

\makeatletter% 
\def\@fnsymbol#1{\ensuremath{\ifcase#1\or  1 \or *\or 2\or  * \or 3\or 4\or * \or \star \or  , \or 
g\or h\or i\else\@ctrerr\fi}}% 
\makeatother% 

\author{\begin{minipage}{.4\textwidth}\center
{\sc Nicolas Asin}\thanks{Institut de statistique, biostatistique et
  sciences actuarielles (ISBA), Voie du Roman Pays 20, 1348~Louvain-la-Neuve,
Belgium,  e-mail:
\url{nicolas.asin@uclouvain.be}}\,\,\thanks{Corresponding  author.}\\[.5ex]\small Universit\'e catholique
de Louvain\\\null\end{minipage}
\begin{minipage}{.4\textwidth}\center
{\sc Jan Johannes}\thanks{Institut f\"ur Angewandte Mathematik, Im Neuenheimer Feld 294,
69120 Heidelberg, Germany, e-mail:
\url{johannes@math.uni-heidelberg.de}}\\[.5ex]\small Ruprecht-Karls-Universit\"at Heidelberg \\\null\end{minipage}}
\date{} 
\title{Adaptive non-parametric estimation\\ in the presence of dependence} 
\begin{document} 
\maketitle 
% --------------------------------------------------------------------
% <<Abstract>>
% --------------------------------------------------------------------
\begin{abstract}
  We consider non-parametric estimation problems in the presence of
  dependent data, notably non-parametric regression with random design
  and non-parametric density estimation.  The proposed estimation
  procedure is based on a dimension reduction.  The minimax optimal
  rate of convergence of the estimator is derived assuming a
  sufficiently weak dependence characterized by fast decreasing mixing
  coefficients.  We illustrate
  these results by considering classical smoothness
  assumptions. However, the proposed estimator requires an optimal
  choice of a dimension parameter depending on certain characteristics
  of the function of interest, which are not known in practice. The
  main issue addressed in our work is an adaptive choice of this
  dimension parameter combining model selection and Lepski’s
  method. It is inspired by the recent work of
  \cite{GoldenshlugerLepski2011}. We show that this data-driven
  estimator can attain the lower risk bound up to a constant provided
  a fast decay of the mixing coefficients.
\end{abstract} 
{\footnotesize
\begin{tabbing} 
\end{tabbing}}
% --------------------------------------------------------------------
% <<Content>>
% --------------------------------------------------------------------
%======================================================================================================================
%                                                                 
% Title: Introduction
% Author: Jan JOHANNES, Nicolas Asin
% 
% Email: jan.johannes@ensai.fr nicolas.asin@uclouvain.be
% Date: %%ts latex start%%[2016-02-01 Mon 02:12]%%ts latex end%%
% Main-TeX-File: t in der form "Name"
%
% ======================================================================================================================
\section{Introduction}\label{s:in}
We study the non-parametric estimation of a functional parameter of
interest $\So$ based on a sample of identically distributed random
variables $Z_1,\dotsc,Z_n$. For convenience, the function of interest $\So$
belongs to the Hilbert space $\Hspace:=\Hspace[0,1]$ of square integrable
real-valued functions defined on $[0,1]$ which is endowed with its usual
inner product $\Hskalar$ and its induced norm $\Hnorm$. In this paper
we study the attainable accuracy of a fully data-driven estimator of
$\So$ for independent as well as dependent observations
$Z_1,\dotsc,Z_n$ from a minimax point of view. The estimator is based
on an orthogonal series approach where the fully data-driven selection
of the dimension parameter is inspired by the recent work of
\cite{GoldenshlugerLepski2011}.  We derive conditions that allow us
to bound the maximal risk of the fully data-driven estimator over
suitable chosen classes $\cF$ for $\So$, which are
constructed flexibly enough to characterize, in particular,
differentiable or analytic functions. Considering two classical non-parametric problems, namely
non-parametric density estimation and non-parametric regression with
random design, we show that these conditions indeed hold true, if the identically distributed 
observations $Z_1,\dotsc,Z_n$ are independent (iid.) or weakly dependent
with sufficiently fast decay of their $\beta$-mixing
coefficients. Thereby, we establish  the rate of convergence of the
 fully data-driven estimator for independent as well as weakly dependent
observations.
Considering iid. observations we show that these rates of convergence
are minimax-optimal for a wide variety of classes $\cF$, and hence
the  fully data-driven estimator is called adaptive. Replacing  the independence
assumption by mixing conditions the rates of convergence of the fully
data-driven estimator are generally slower. A comparison, however,
allows us to state conditions on the mixing coefficients which ensure
that the fully data-driven estimator still attains the minimax-optimal
rates for a wide variety of classes of $\cF$, and hence, is adaptive.
The adaptive non-parametric estimation based on weakly dependent
observations of either a
density or a regression function has been consider by  \cite{TribouleyViennet1998},
\cite{ComteMerlevede2002}, \cite{ComteRozenholc2002}, \cite{GannazWintenberger2010},
\cite{ComteDedeckerTaupin2008} or \cite{BertinKlutchnikoff2014}, to
name but a view. However, our conditions to derive rates of convergence of
the fully-data driven estimator can be verified for both, non-parametric density estimation and non-parametric regression with
random design. Thereby, we think that these conditions provide a promising starting point to deal with more complex non-parametric
models, as for example, errors in variables model. 

The paper is organized as follows: in Section \ref{s:mo} we introduce
our basic assumptions, define the class $\Soc$ and develop the data-driven orthogonal series
estimator. We present key arguments of the proofs while technical
details are postponed to the Appendix.
We show, in Section \ref{s:id}, the minimax-optimality of the
data-driven estimator of a density as well as a regression function
based on iid. observations. In Section \ref{s:dd} we briefly review
elementary dependence notions and present standard coupling arguments. 
Considering again the non-parametric estimation of a density as well as a
regression function we derive mixing conditions such that the fully data-driven estimator based on
dependent observations  can attain the minimax-rates for independent
data. Finally, considering  the framework used by \cite{GannazWintenberger2010} and
\cite{BertinKlutchnikoff2014} results of a simulation study are
reported in Section \ref{s:sim} which allow to compare the finite sample performance of different
data-driven estimators of a density as well as a regression function
given independent or dependent observations.

%%% Local Variables: 
%%% mode: latex
%%% TeX-master: "_0DP_NPE_dep"
%%% End: 

%======================================================================================================================
%                                                                 
% Title: Model assumptions and notations
% Author: Jan JOHANNES, Nicolas Asin
% 
% Email: jan.johannes@ensai.fr nicolas.asin@uclouvain.be
% Date: %%ts latex start%%[2016-02-01 Mon 02:15]%%ts latex end%%
% Main-TeX-File: t in der form "Name"
%
% ======================================================================================================================
\section{Model assumptions and notations}\label{s:mo}
% --------------------------------------------------------------------
% <<Assumptions \label{s:mo:no}>>
% --------------------------------------------------------------------
\subsection{Assumptions and notations}\label{s:mo:no}
We construct an estimator of the unknown function $\So$ using an
orthogonal series approach. The estimation of $\So$ is based on a
dimension reduction which we elaborate in the following. Let us
specify an arbitrary orthonormal system $\set{\bas_j}_{j=1}^\infty$ of
$\Hspace$. We denote by $\Proj_\Phi$
and $\Proj^\perp_\Phi$ the orthogonal projections on the linear subspace $\Phi$ spanned by this orthonormal system and its orthogonal complement $\Phi^{\perp}$ in $\Hspace$, respectively. Consequently, any function
$h\in\Phi$ admits an expansion $h =\sum_{j=1}^\infty
\fou{h}_j\bas_j$ as a generalised Fourier series with coefficients
$\fou{h}_j:=\HskalarV{h,\bas_j}$ for $j\geq1$. The unknown function
$\So\in\Hspace$ is thereby uniquely determined by its coefficients
$(\fSo_j)_{j\geq1}$, or $\fSo$ for short, and $\Proj^\perp_\Phi\So$. In what follows $\Proj^\perp_\Phi\So$ is know in advance while the sequence of  coefficients $\fSo$ has to be estimated.  Given a dimension parameter $\Di\geq1$ we have
the subspace $\Dz_\Di$ spanned by the first $\Di$ basis functions
$\set{\bas_j}_{j=1}^\Di$ at our disposal. For abbreviation, we denote by $\Proj_\Di$ and $\Proj^\perp_\Di$ the orthogonal projections on the linear subspace $\Dz_\Di$ and its orthogonal complement $\Dz_\Di^{\perp}$ in $\Phi$, respectively. We consider the orthogonal
projection $\DiSo:=\Proj^\perp_\Phi\So+\Proj_{\Di}\So$ of $\So$
admitting the expansion $\Proj_\Di\DiSo=\sum_{j=1}^\Di\fSo_j\bas_j$ and its associated approximation error
$\bias(\So):=\HnormV{\DiSo-\So}=\HnormV{\Proj^\perp_\Di\So}$ where
$\bias(\So)$ tends to zero as $\Di\to\infty$ for all $\So\in \Hspace$
due to the dominated convergence theorem. We consider an orthogonal
series estimator $\hDiSo[\Di]$ by replacing, for $j=1,\dotsc,\Di$, the
coefficient $\fSo_j$ by its empirical counterpart $\hfSo_j$, that is,
$\hDiSo[\hDi]=\sum_{j=1}^\Di\hfSo_j\bas_j$. The attainable accuracy of
the  proposed estimator of $\So$  are basically determined by a priori
conditions on $\So$. These conditions are often expressed in the form
$\So\in\Soc$, for a suitably chosen class $\Soc\subset\Hspace$. This class $\Soc$ reflects prior
 information on the function $\So$, e.g., its level of smoothness, 
and will be constructed flexibly enough to characterize, in particular, differentiable or analytic functions.
% --------------------------------------------------------------------
% <<Minimal regularity assumption>>
% --------------------------------------------------------------------
We determine the class
$\Soc$ by means of a weighted norm in $\Phi$. Given the
orthonormal basis $\set{\bas_j}_{j=1}^\infty$ of $\Phi$ and a strictly positive  sequence of weights
$(\Sow_j)_{j\geq1}$, or $\Sow$ for short,  we define for
$h\in\Phi$ the weighted norm
$\wnormV{h}^2:=\sum_{j\in\Nz}\Sow_j^{-1}\fou{h}_j^2$. Furthermore, we denote by $\Phi_{\Sow}$ and  $\Phi_\Sow^\Sor$ for a constant $\Sor>0$, the completion of $\Phi$ with respect to
$\wnorm$ and the ellipsoid $\Phi_\Sow^\Sor:=\set{h\in \Phi : \wnormV{ h}^2\leq \Sor^2}$. Obviously, for a non-increasing sequence $\Sow$ the class
$\Phi_{\Sow}^\Sor$ is a subspace of $\Phi$. Here and subsequently, we  assume
that there exist a monotonically non-increasing and strictly positive
sequence of weights $\Sow$ tending to zero and a constant $\Sor>0$ such  that 
 the function of interest $\So$ belongs to the $\Socwr:=\set{f\in \Hspace : \Proj_{\Phi}f \in\Phi_{\Sow}^\Sor}$.
We may emphasize that for any
$\So\in\Socwr$, $\bias^2(\So)=\sum_{j>m}(\Sow_j/\Sow_j)\fSo_j^2\leq
\Sow_\Di\wnormV{\Proj_{\Phi}\So}^2\leq \Sow_\Di \Sor^2$ which we use in the sequel
without further reference.

% --------------------------------------------------------------------
% <<Basis>>
% --------------------------------------------------------------------
Further denote by $\InormV{h}$  as usual the  $\Ispace$ norm of a function
$h\in\Hspace$.  We require in the sequel that the orthonormal system $\set{\bas}_j$ and the sequence $\ga$ satisfy the following assumptions.
\begin{enumerate}[label={\textbf{(A\arabic*)}},ref={\textbf{(A\arabic*)}}]
\item\label{mo:no:as:ba:i} There exists a finite constant
  $\maxnormsup\geq1$ such that $\InormV{\sum_{j=1}^\Di\bas_j^2}\leq
  \maxnormsup^2\Di$ for all $\Di\in\Nz$. 
\item\label{mo:no:as:ba:ii} The sequence $\ga$ is monotonically decreasing with limit zero and there exists a finite constant
  $\gA\geq1$ such that $\InormV{\sum_{j\geq1}\Sow_j\bas_j^2}\leq \gA^2$.
\end{enumerate}
According to Lemma 6 of \cite{BirgeMassart1997} assumption \ref{mo:no:as:ba:i} is
exactly equivalent to following property: there exists a positive constant $\maxnormsup$ such that for any
  $h\in\Dz_\Di$ holds $\InormV{h}\leq \maxnormsup\sqrt{\Di} \HnormV{h}$.
Typical example are bounded basis, such as the trigonometric
basis, or basis satisfying the assertion, that there exists a positive constant $C_\infty$ such that for any
$(c_1,\dotsc,c_\Di)\in\Rz^\Di$,  $\InormV{\sum_{j=1}^\Di
  c_j\bas_j}\leq C_\infty\sqrt{\Di}\absV{c}_\infty$ where
$\absV{c}_\infty=\max_{1\leq j\leq\Di}\absV{c_j}$.
\cite{BirgeMassart1997} have shown that the last property is satisfied
for piecewise polynomials, splines and wavelets. On the other hand
side, in the case of a bounded basis the property
\ref{mo:no:as:ba:ii} holds for any summable weight sequence $\Sow$,
i.e.,  $\pabsV{\Sow}:=\sum_{j\geq1}\Sow_j<\infty$. More generally,
under \ref{mo:no:as:ba:i} the additional assumption $\sum_{j\geq1}j\Sow_j<\infty$ is
sufficient to ensure \ref{mo:no:as:ba:ii}. Furthermore, under
\ref{mo:no:as:ba:ii} the elements of $\Phi_\Sow^\Sor$ are bounded
uniformly, that is $\InormV{h}^2\leq\InormV{\sum_{j\geq1}\Sow_j\bas_j^2}\wnormV{h}^2\leq\gA^2\Sor^2<\infty$  for any $h\in\Phi_\Sow^\Sor$.

% --------------------------------------------------------------------
% <<Observations \label{s:mo:ob}>>
% --------------------------------------------------------------------

\subsection{Observations}\label{s:mo:ob}
In this work we focus on two models, namely non-parametric regression
with random design and non-parametric density estimation. The important point to note here is  that in each model the identically distributed (i.d.) observations $Z_1,\dotsc, Z_n$ satisfy  $\Ex\psi_j(\Ob_i)=\fSo_j$ for a certain function $\psi_j$ , $j\geq 1$. Therefore, given an i.d. sample
$\{Z_i\}_{i=1}^n$, it is natural to
consider the estimator $\hfSo_j=n^{-1}\sum_{i=1}^n\psi_j(Z_i)$ of $\fSo_j$.\\
\paragraph{Non-parametric regression.}
A common problem in statistics is to investigate the dependence of a
real random variable $\rOb$ on the variation of an explanatory random
variable $\rRe$. For convenience, the regressor $\rRe$ is supposed to
be uniformly distributed on the interval $[0,1]$, i.e., $\rRe\sim
\cU[0,1]$. In this paper, the dependence of $\rOb$ on $\rRe$ is
characterised by $\rOb=\So(\rRe)+\rNoL\rNo$, for $\rNoL>0$, where $\So\in\Hspace$ is an unknown function and  $\rNo$ is a centred and standardised error term. Furthermore, we
suppose that $\rNo$ and $\rRe$ are independent. Keeping in mind the expansion
$\So=\sum_{j=1}^\infty\fSo_j\bas_j$ with respect to the basis
$\set{\bas_j}_{j=1}^\infty$ we observe that $\fSo_j=\Ex(\psi_j(\rOb,\rRe))$ with $\psi_j(\rOb,\rRe)=\rOb\bas_j(\rRe)$
for all $j\geq1$.
\paragraph{Non-parametric density estimation.} Let $\dOb$ be a random
variable taking its values in $[0,1]$ and admitting a density $\So$
which belongs to the set  $\cD$ of all densities with
support included in $[0,1]$. We focus on the non-parametric estimation
of the density $\So$ if it is in addition square integrable, i.e.,
$\So\in\Hspace$. For convenient
notations, let $\1(t):=1,$ $t\in[0,1]$  and
$\{\1\}\cup\{\bas_j\}_{j=1}^\infty$ be an orthonormal basis of
$\Hspace$.  Keeping in mind that $\So$ is a density, it admits an
expansion $\So= \1+ \sum_{j=1}^\infty \fSo_j\bas_j$ where
$\fSo_j=\Ex[\bas_j(\dOb)]$
for all $j\geq1$. 
In this context we notice that $\Phi^{\perp}$ is spanned by $\1$. Since $\So$ is a density function  we have $\Proj^\perp_\Phi\So=\1$, which is obviously known in advance.

% --------------------------------------------------------------------
% <<Methodology \label{s:mo:me}>>
% --------------------------------------------------------------------
\subsection{Methodology and background}\label{s:mo:me}
For the simplicity of the presentation, we assume throughout this section that $\So\in\Phi$, that is $\Proj^\perp_\Phi\So=0$. The orthogonal projection $\DiSo=\sum_{j=1}^\Di\fSo_j\bas_j$ at hand let us define an
orthogonal series estimator by replacing  for $j=1,\dotsc, \Di$ the unknown coefficient
$\fSo_j$ by its empirical mean
$\hfSo_j=n^{-1}\sum_{i=1}^n\psi_j(\Ob_i)$, that is,
$\hDiSo=\sum_{j=1}^\Di\hfSo_j\bas_j$.  We shall assess the accuracy
 of the estimator $\hDiSo$  by its maximal integrated mean
 squared error with respect to the class $\Soc$, that is $\Rif{\hSo}{\Soc}:=\sup_{\So\in\Soc}\Ex\HnormV{\hSo-\So}^2$.
 Considering identically and
independent distributed (iid.) observation obeying the two models, non-parametric
regression  and density estimation, we
derive a lower bound for the maximal risk over $\Soc$ for all estimators and show that
it provides up to a  positive constant $C$ possibly depending on the class  $\Soc$ also an upper bound for the
maximal risk over $\Soc$ of the orthogonal series estimator
$\hDiSo[\oDi]$ with suitable chosen dimension parameter $\oDi\in\Nz$, i.e.,
 \begin{equation*}
   \Rif{\hDiSo[\oDi]}{\Soc}\leq C\cdot \inf_{\tSo}\Rif{\tSo}{\Soc}
 \end{equation*}
where the infimum is taken over all estimators of $\So$. We thereby
prove the minimax optimality of the estimator $\hDiSo[\oDi]$. 
Obviously, if the observations are independent or sufficiently weak dependent   there exists a finite constant $C>0$
  possibly depending on the class $\Socwr$ such that $\sup_{\So\in\Socwr}\sum_{j=1}^m
\Var(\hfSo_j)\leq C \Di n^{-1}$ for all $\Di,n\geq1$.
 From the Pythagorean formula
we obtain the identity  $\HnormV{\hDiSo-\So}^2=
\HnormV{\hDiSo-\DiSo}^2+\bias^2(\So)$ and, hence together with
$\bias^2(\So)\leq\Sow_\Di\Sor^2$ for all $\So\in\Socwr$  follows
\begin{equation}\label{mo:me:ri:ub}
	\Rif{\hDiSo}{\Socwr}\leq\Sow_\Di\Sor^2+C \Di
        n^{-1}=(\Sor^2+C)\max(\Sow_\Di,\Di n^{-1}).
\end{equation}
The upper bound in the last display depends on the dimension parameter $\Di$ and hence
by choosing an optimal value $\oDi$  the upper bound  will be minimized which we formalize next.
For  a sequence $(a_m)_{m\geq1}$ with minimal
value in $A$ we set $\argmin\nolimits_{m\in A}\set{a_m}:=\min\{m:a_m\leq a_{k},\forall k\in A\}$ and define for all  $n,\Di\geq1$
\begin{multline}\label{mo:me:de:ra}
  \mRa:=\mRa(\Sow):=[\Sow_\Di\vee\Di n^{-1}]:=\max(\Sow_\Di,\Di
  n^{-1}), \\ \oDi:=\oDi(\Sow):=\argmin_{\Di\in\Nz}\set{\mRa}\quad \text{ and }\quad \oRa:=\oRa(\Sow):=\mRa[\oDi]=\min_{\Di\in\Nz}\mRa.
\end{multline}
From \eqref{mo:me:ri:ub} we deduce that $\Rif{\hDiSo[\oDi]}{\Socwr}\leq (\Sor^2+C)\oRa$ for all $n\geq 1$.  Moreover if it is possible to show that
  $\oRa$  provides up to a constant also a lower bound of $\Rif{\hDiSo[\oDi]}{\Socwr}$ then the
  estimator $\hDiSo[\oDi]$ with optimal chosen $\oDi$ is
  minimax rate-optimal. However, $\oDi$ depends on the unknown
  regularity of $\So$ and hence we will introduce below  a data-driven
  procedure to select the dimension parameter. Let us first briefly illustrate the last definitions by stating the order of  $\oDi$ and $\oRa$ for typical choices of the sequence $\ga$. 
 \begin{illu}\label{illus} We will illustrate all our results considering the following two configurations for the sequence $\ga$. Here and subsequently, we use for two strictly positive sequences $(x_n)_{n\geq1}$, $(y_n)_{n\geq1}$ the notation $x_n \asymp y_n$ if $(x_n/y_n)_{n\geq1}$ is bounded away both from zero and infinity. Let,
 \begin{enumerate}
\item[(p)]\label{ill:expo} $\ga_j=|j|^{-2p}$, $j\geq 1$, with $p>1$, then
$\oDi\asymp n^{-1/(2p+1)}$and $\oRa\asymp n^{-2p/(2p+1)}$;
\item[(e)]\label{ill:poly} $\ga_j=\exp(|j|^{-2p})$, $j\geq 1$,  with $p>0$, then $\oDi\asymp (\log(n))^{1/2p}$ and $\oRa\asymp  n^{-1}(\log(n))^{1/2p}$.
\end{enumerate}We note that the assumption \ref{mo:no:as:ba:ii} and $(\oRa)^{-1}\min(\Sow_{\oDi}, \oDi n^{-1})\asymp 1$ hold true in both cases.
 \end{illu}

% --------------------------------------------------------------------
% <<Adaptive Estimation>>
% --------------------------------------------------------------------
Our selection method of the dimension parameter is inspired by the work of
\cite{GoldenshlugerLepski2011} and combines the techniques of model
selection and Lepski's method. We determine the dimension parameter
among a collection of admissible values by minimizing a penalized
contrast function. To this end, for all $n\geq1$ let  $(\pen[1],...,\pen[\DiMa])$ be a subsequence of non-negative and non-decreasing penalties. We select $\tDi$ among the collection $\set{1,\dotsc,\DiMa}$ such that: 
\begin{equation}\label{mo:me:de:wtm}
\tDi=\argmin_{1\leq m\leq \DiMa}\set{\contr+\pen}
\end{equation}
where the contrast is defined by $\contr:=\max_{m\leq k\leq \DiMa}
\set{\HnormV{\hDiSo-\hDiSo[k]}^2-\pen[ k]}$ for all $1\leq m\leq \DiMa$. The data-driven estimator is now given by $\hDiSo[\tDi]$ and our
aim is to prove an upper bound for its maximal risk
$\Rif{\hDiSo[\widetilde{m}]}{\Socwr}$. We outline next the main ideas
of the proof and introduce conditions which we will show below hold
indeed true for the two considered non-parametric estimation problems. A key argument is the next lemma due to \cite{ComteJohannes2012}.
\begin{lem}\label{mo:me:le:ms}
If $\left(\pen[1],\dotsc,\pen[\DiMa]\right)$ is a non-decreasing subsequence
and $1\leq \Di\leq \DiMa$, then
\begin{eqnarray*}
\HnormV{\hDiSo[\widetilde{\Di}]-\So}^2\leq 85 \max(\bias^2(\So),\pen) +42 \max_{\Di\leq k\leq \DiMa}\vectp{ \HnormV{\hDiSo[k]-\DiSo[k]}^2 - \pen[k]/6}
\end{eqnarray*}
where $(x)_+:=\max(x,0)$.
\end{lem}
\noindent  Keeping in mind that $\bias^2(\So)\leq\Sow_\Di\Sor^2$ for all
$\So\in\Socwr$ we impose  the following condition.
\begin{enumerate}[label={\textbf{(C\arabic*)}},ref={\textbf{(C\arabic*)}}]
\item\label{mo:me:as:est:A1} There exists a finite constant $\delta>0$
  possibly depending on the class $\Socwr$ such that
  $\sup_{\So\in\Socwr}\max_{1\leq \Di\leq \DiMa} \{\pen/\Di\}\leq \delta n^{-1}$ for all $n\geq1$.
\end{enumerate}
Under condition \ref{mo:me:as:est:A1} and employing and $\mRa[\Di]=\max(\Sow_\Di,\Di n^{-1})$ we have due to Lemma
\ref{mo:me:le:ms} that for all $1\leq \Di\leq \DiMa$
\begin{equation}\label{an:ri:dd:e1}
\sup_{\So\in\Socwr}\Ex\HnormV{\hDiSo[\tDi]-\So}^2\leq
85(\Sor^2\vee\delta) \mRa[\Di] +42 \sup_{\So\in\Socwr}\Ex \max_{\Di\leq k\leq \DiMa}\vectp{ \HnormV{\hDiSo[k]-\DiSo[k]}^2 - {\pen[k]}/{6}}.
\end{equation}
Keeping mind that $\oRa=\min_{\Di\in\Nz}\mRa[\Di]=\mRa[\oDi]$ where $\oDi=\argmin_{\Di\in\Nz}\{\mRa[\Di]\}$ realises a
variance-squared-bias compromise among all values in $\Nz$. Considering  
 the subset $\{1,\dotsc,\DiMa\}$ rather than $\Nz$ we
have trivially  $\oRa=\min_{1\leq\Di\leq\DiMa}\mRa$
if $\oDi\leq\DiMa$. On the other hand,
since $\oRa=o(1)$ as $n\to\infty$ there exists  $n_{\aSy}\in\Nz$ with
$\oRa[n_{\aSy}]\leq 1$ for all $n\geq n_{\aSy}$ which in turn implies
$\oDi\leq n$ for all $n\geq n_{\aSy}$. Indeed, $\oDi n^{-1}\leq
\oRa\leq \oRa[{n_{\aSy}}]\leq 1$ for all  $n\geq n_{\aSy}$ implies that
$\oDi\leq n$. Thereby, we have
$\oRa=\min_{1\leq\Di\leq\DiMa}\mRa$ for all $n\geq n_{\aSy}$. Consequently, from \eqref{an:ri:dd:e1} follows for all $n\geq
n_{\aSy}$ 
\begin{equation}\label{an:ri:dd:e2}
\Rif{\hDiSo[\tDi]}{\Socwr}\leq85(\delta\vee\Sor^2)\mRa[\oDi]+42\sup_{\So\in\Socwr}\Ex\max_{\oDi\leq \Di\leq n}\vectp{\HnormV{\hDiSo-\DiSo}^2 -\pen/6}.
\end{equation}
The second right hand side (rhs.) term in the last display we bound
using the next condition.
\begin{enumerate}[label={\textbf{(C\arabic*)}},ref={\textbf{(C\arabic*)}}]\addtocounter{enumi}{1}
\item\label{mo:me:as:est:A2} 
There exists a finite constant $\Delta>0$ possibly depending on the class $\Socwr$ such that
$\sup_{\So\in\Socwr}\Ex\set{\max_{\oDi\leq \Di\leq\DiMa}\vectp{\HnormV{\hDiSo-\DiSo}^2 -1/6\pen}}\leq \Delta n^{-1}$ for all $n\geq1$.
\end{enumerate}
From \eqref{an:ri:dd:e2} together with 
\ref{mo:me:as:est:A2} it follows that
\begin{equation}\label{mo:me:as:est:e} 
\Rif{\hDiSo[\tDi]}{\Socwr}\leq 85(\delta\vee\Sor^2)\mRa[\oDi] +42\Delta
n^{-1},\quad\mbox{for all }n\geq1.
\end{equation}
The next assertion is an
immediate consequence and hence we omit its proof.
\begin{prop}\label{an:pr:dd}
Let \ref{mo:me:as:est:A1} and \ref{mo:me:as:est:A2} be
satisfied, then  for all $n\geq n_{\aSy}$ holds
\begin{equation*}
\Rif{\hDiSo[\tDi]}{\Socwr}\leq 85(\delta\vee\Sor^2)\oRa +42\Delta
n^{-1}\leq 127(\delta\vee\Sor^2\vee\Delta)\,\oRa,\quad\mbox{for all }n\geq n_{\aSy}
\end{equation*}
where  $n_{\aSy}\in\Nz$ satisfies  $\oRa[n_{\aSy}]\leq 1$.
\end{prop}
The last assertion establishes an upper risk bound of the estimator $\hDiSo[\tDi]$.
We call  $\hDiSo[\tDi]$ partially data-driven  if the
sequence of penalty terms still depend on unknown quantities which
however, can be estimated. In this situation, let $\hpen$ be an estimator of
$\pen$ such that the  subsequence of
 penalties
$\left(\hpen[1],\dotsc,\hpen[\DiMa]\right)$ is non-negative and
non-decreasing. 
The  dimension parameter $\hDi$ is then selected  among the collection
$\set{1,\dotsc,\DiMa}$ as follows
\begin{equation}\label{mo:me:de:whm}
\hDi=\argmin_{1\leq m\leq \DiMa}\set{\hcontr+\hpen}
\end{equation}
where the contrast is defined by $\hcontr:=\max_{m\leq
  k\leq \DiMa} \set{\HnormV{\hDiSo-\hDiSo[k]}^2-\hpen[k]}$ for all $1\leq
m\leq \DiMa$. Following line by line the proof of  Lemma
\ref{mo:me:le:ms} we obtain
\begin{equation}\label{mo:me:le:ms:hpen}
  \HnormV{\hDiSo[\hDi]-\So}^2\leq 85 \max(\bias^2(\So),\hpen) +42
  \max_{\Di\leq k\leq \DiMa}\vectp{ \HnormV{\hDiSo[k]-\DiSo[k]}^2 - \hpen[k]/6}.
\end{equation}
Keeping the last bound in mind we decompose the risk with respect to
an event on which  the  quantity $\hpen$ is close to its theoretical
counterpart $\pen$. More precisely, define the event
\begin{equation}\label{mo:me:le:ms:omega}
	\Omega=\set{\pen\leq\hpen\leq 3\pen; \quad\forall 1\leq \Di\leq \DiMa }
\end{equation}
and denote by $\Omega^c$ its complement. Let us consider the following decomposition for the maximal risk :
\begin{equation}\label{an:ri:dd:e3}
\Rif{\hDiSo[\hDi]}{\Socwr}= \sup_{\So\in\Socwr}\Ex\left(\indicset{\Omega}\HnormV{\hDiSo[\hDi]-\So}^2\right)+\sup_{f\in\Socwr}\Ex\left(\indicset{\Omega^c}\HnormV{\hDiSo[\hDi]-\So}^2\right)
\end{equation}
where we bound the two rhs. terms separately. 
\begin{lem}\label{an:le:dd:re}
Under Assumption \ref{mo:me:as:est:A1} and \ref{mo:me:as:est:A2} we
have 
\begin{equation*}\sup_{f\in\Socwr}\Ex\left(\indicset{\Omega^c}\HnormV{\hDiSo[\hDi]-\So}^2\right)
  \leq \Delta n^{-1} + \{\Sor^2  +\delta \} P(\Omega^c).
\end{equation*}
\end{lem}
Due to the last assertion the second rhs. term in \eqref{an:ri:dd:e3} is bounded up to a constant by  $n^{-1}$ if the probability $P(\Omega^c)$ is sufficiently small, which we precize next.
\begin{enumerate}[label={\textbf{(C\arabic*)}},ref={\textbf{(C\arabic*)}}]\addtocounter{enumi}{2}
\item\label{mo:me:as:est:A3} 
There exists a finite constant $\kappa>0$ possibly depending on the class $\Socwr$ such that
$\sup_{f\in\Socwr} \DiMa P(\Omega^c)\leq \kappa $ for all $n\geq1$.
\end{enumerate}
Considering the
first rhs. term in \eqref{an:ri:dd:e3}  we  employ the inequality
\eqref{mo:me:le:ms:hpen}, that is
\begin{equation}\label{mo:me:le:ms:hpen2}
	\HnormV{\hDiSo[\hDi]-\So}^2\indicset{\Omega}
		\leq255 \max(\bias^2(\So),\pen) +42 \max_{\Di\leq k\leq \DiMa}\vectp{ \HnormV{\hDiSo[k]-\DiSo[k]}^2 -\pen[k]}.
\end{equation}
Following now line by line the proof of Proposition \ref{an:pr:dd} the
next assertion is an immediate consequence of Lemma \ref{an:le:dd:re}, the condition
\ref{mo:me:as:est:A3} and \eqref{mo:me:le:ms:hpen2} and we
omit its proof.
\begin{prop}\label{an:pr:dd:hpen}
Under \ref{mo:me:as:est:A1}, \ref{mo:me:as:est:A2} and
\ref{mo:me:as:est:A3}  holds
\begin{equation*}
\Rif{\hDiSo[\hDi]}{\Socwr}\leq 255(\delta\vee\Sor^2)\oRa +43\Delta
n^{-1} + \kappa (\Sor^2+\delta) n^{-1}\leq (298+2\kappa)(\delta\vee\Sor^2\vee\Delta)\,\oRa,\;\mbox{for all }n\geq n_{\aSy}
\end{equation*}
where  $n_{\aSy}\in\Nz$ satisfies  $\oRa[n_{\aSy}]\leq 1$.
\end{prop}

Considering the two models, namely non-parametric density estimation and non-parametric regression, we will show that the conditions \ref{mo:me:as:est:A1} and \ref{mo:me:as:est:A2} and
\ref{mo:me:as:est:A3} are verified. Thereby, an upper bound for the data-driven estimator $\hDiSo[\tDi]$ and $\hDiSo[\hDi]$ can be deduced from Proposition \ref{an:pr:dd} and \ref{an:pr:dd:hpen}, respectively.

%%% Local Variables: 
%%% mode: latex
%%% TeX-master: "_0DP_NPE_dep"
%%% End: 

%======================================================================================================================
%                                                                 
% Title: Independent observations
% Author: Jan JOHANNES, Nicolas Asin
% 
% Email: jan.johannes@ensai.fr nicolas.asin@uclouvain.be
% Date: %%ts latex start%%[2016-02-01 Mon 02:15]%%ts latex end%%
% Main-TeX-File: t in der form "Name"
%
% ======================================================================================================================
\section{Independent observations}\label{s:id}
In this section we suppose that the identically distributed $n$-sample $\{\Ob_i\}_{i=1}^n$
consists of independent random variables. Considering the two non-parametric estimation problems we will show that $\oRa$ given in \eqref{mo:me:de:ra} provides a lower bound of the maximal risk $\Rif{\tSo}{\Socwr}$ for all possible estimators $\tSo$. On the other hand side, $\oRa$ will provide also an upper bound up to a constant of the maximal risk of the orthogonal series estimator $\hSo_{\oDi}=\sum_{j=1}^{\oDi}\hfSo_j\bas_j$ with optimally chosen dimension parameter. Thereby, $\oRa$ is the minimax-optimal rate of convergence and the estimator $\hSo_{\oDi}$ is minimax-rate optimal. However, the dimension parameter $\oDi$ depends on the class of unknown function. In a second step we will show by applying Proposition \ref{an:pr:dd} and \ref{an:pr:dd:hpen}, respectively, that the data-driven estimator $\hDiSo[\tDi]$ and $\hDiSo[\hDi]$ can attain the minimax-optimal rate of convergence. The key argument to verify the condition \ref{mo:me:as:est:A2} is the following inequality, which is due to \cite{Talagrand1996}
and can be found for example in \cite{KleinRio2005}.

\begin{lem}(Talagrand's inequality)\label{id:ka:l:talagrand} Let
  $\Ob_1,\dotsc,\Ob_k$ be independent $\cZ$-valued random variables and let $\overline{\nu_t}=k^{-1}\sum_{i=1}^k\left[\nu_t(\Ob_i)-\Ex\left(\nu_t(\Ob_i)\right) \right]$ for $\nu_t$ belonging to a countable class $\{\nu_t,t\in\cT\}$ of measurable functions. Then,
\begin{equation*}
	\Ex\vectp{\sup_{t\in\cT}|\overline{\nu_t}|^2-6H^2}\leq C \left[\frac{v}{k}\exp\left(\frac{-kH^2}{6v}\right)+\frac{h^2}{k^2}\exp\left(\frac{-K k H}{h}\right) \right]
\end{equation*}
with numerical constants $K=({\sqrt{2}-1})/({21\sqrt{2}})$ and $C>0$ and where
\begin{equation*}
	\sup_{t\in\cT}\sup_{x\in\cZ}|\nu_t(x)|\leq h,\qquad \Ex\left[\sup_{t\in\cT}|\overline{\nu_t}|\right]\leq H,\qquad \sup_{t\in\cT}\frac{1}{k}\sum_{i=1}^k \Var(\nu_t(\Ob_i))\leq v.
\end{equation*}
\end{lem}
 \begin{rem}\label{id:ka:rem}
 Let us briefly reconsider the orthogonal series estimator. Introduce further the unit ball $\Bz_m:=\set{h\in\Dz_m:\HnormV{h}\leq1}$
 contained in the subspace $\Dz_m=\lin\set{\bas_{1},\dotsc,\bas_m}$
 which is a countable set of functions. Moreover, set
 $\overline{\nu_t}=n^{-1}\sum_{i=1}^n\left[\nu_t(\Ob_i)-\Ex\left(\nu_t(\Ob_i)\right)
 \right]$ and $\nu_t(\Ob)=\sum_{j=1}^\Di\fou{t}_j\psi_j(\Ob)$, then we have
\begin{equation*}
	\HnormV{\hSo_\Di-\So_\Di}^2=\sup_{t\in\Bz_m}|\skalarV{\hSo_\Di-\So_\Di,t}|^2=\sup_{t\in\Bz_m}|\sum_{j=1}^\Di(\hfSo_j-\fSo_j)\fou{t}_j|^2=\sup_{t\in\Bz_m}|\overline{\nu_t}|^2.
\end{equation*}
The last identity provides the necessary argument to link the
condition  \ref{mo:me:as:est:A2} and Talagrand's inequality.  Moreover we will suppose that the ONS $\set{\bas_j}_{\j\in\Nz}$ and the weight sequence $\ga$ used to construct the ellipsoid $\Socwr$ satisfy the assumptions \ref{mo:no:as:ba:i} and \ref{mo:no:as:ba:ii}.
\end{rem}
\subsection{Non-parametric density estimation}\label{s:id:d}
In this paragraph we suppose that the identically distributed $n$-sample $\{\dOb_i\}_{i=1}^n$
consists of independent  random variables admitting a common density $\So$ which belongs to the set  $\cD$ of all densities with support included in $[0,1]$.
\begin{prop}[Upper bound]\label{id:d:p:ub} Let $\{\dOb_i\}_{i=1}^n$ be an iid. $n-$sample. Under the assumption
\ref{mo:no:as:ba:i} holds 
\begin{equation}\label{id:d:p:ub:e1}
\Rif{\hDiSo[\oDi]}{\Socwr\cap\cD}\leq (\maxnormsup^2+\Sor^2)\;\oRa,\quad\mbox{for all }n\geq 1.
\end{equation}
\end{prop}
\begin{prop}[Lower bound]\label{id:d:p:lb}
Suppose $\{\dOb_i\}_{i=1}^n$ is an iid. $n-$sample. Let the assumption \ref{mo:no:as:ba:ii} holds true and assume further that\begin{equation}\label{id:d:p:lb:e1}
0<\eta:=\inf_{n\geq1}\{(\oRa)^{-1}\min(\Sow_{\oDi}, \oDi n^{-1})\leq 1
\end{equation}
then for all $n\geq 2$ we have 
\begin{equation}\label{id:d:p:lb:e}
\inf_{\tSo}\Rif{\tSo}{\Socwr\cap\cD} \geq \tfrac{\eta}{8}\, \min(\Sor-1,(4\gA)^{-1})\, \oRa
\end{equation}
where the infimum is to be taken over all possible estimators $\tSo$ of $\So$.
\end{prop}
Note that in the configurations considered in the Illustration \ref{illus} the additional condition \eqref{id:d:p:lb:e1} is always satisfied. Comparing the upper bound \eqref{id:d:p:ub:e1} and the lower bound \eqref{id:d:p:lb:e} we have shown that $\oRa$ is the minimax-optimal rate of convergence and the estimator $\hDiSo[\oDi]$ is minimax-optimal.

\paragraph{Fully data-driven estimator.} We consider the fully-data-driven estimator $\hDiSo[\widetilde{m}]$ where $\widetilde{m}$ is defined in \eqref{mo:me:de:wtm} with $\pen:= 36 \maxnormsup^2\Di n^{-1}$ which satisfies trivially the condition \ref{mo:me:as:est:A1}.
The proof of the next Proposition is based on Talagrand's inequality (Lemma \ref{id:ka:l:talagrand}).
\begin{prop}\label{id:d:p:co} Let $\{\dOb_i\}_{i=1}^n$ be an iid. $n-$sample. Suppose that the assumptions \ref{mo:no:as:ba:i} and
  \ref{mo:no:as:ba:ii} are satisfied. There
  exists a numerical constant $C>0$ such that
\begin{equation*}
\sup_{\So\in\Socwr\cap\cD}\Ex\set{\max_{1\leq \Di\leq n}\vectp{\HnormV{\hDiSo-\DiSo}^2
    -6\maxnormsup^2\Di n^{-1}}}\leq C n^{-1} \maxnormsup^2\zeta(\Sor\gA/\maxnormsup^2)
\end{equation*}
where $\zeta(x):=1+x\sum_{m=1}^\infty \exp(-m/(6\sqrt{2}x))$, for any $x>0$.
\end{prop}
By using the definition of the penalty term the last Proposition implies that the condition \ref{mo:me:as:est:A2} is satisfied. Thereby, the
next assertion is an immediate consequence of  Proposition \ref{an:pr:dd}
and we omit its proof.

\begin{theo}\label{id:d:t:ad}Suppose $\{\dOb_i\}_{i=1}^n$ is an iid. $n-$sample.  Let \ref{mo:no:as:ba:i} and
  \ref{mo:no:as:ba:ii} be satisfied. Select the dimension parameter
  $\tDi$ as given by \eqref{mo:me:de:wtm} with  $\pen:= 36 \maxnormsup^2\Di n^{-1}$. There
  exists a numerical constant $C>0$ such that for
  all $n\geq n_{\aSy}$ with $\oRa[n_{\aSy}]\leq 1$ we have
\begin{eqnarray*}
  \Rif{\hDiSo[\tDi]}{\Socwr\cap\cD}\leq C \,[\Sor\vee\maxnormsup^2\vee \maxnormsup^2\zeta(\Sor\gA/\maxnormsup^2)]\,\oRa.
\end{eqnarray*}
\end{theo}
The last assertion establishes the minimax-optimality of the
data-driven estimator $\hDiSo[\tDi]$ over all classes
$\Socwr\cap\cD$ where $\Sow$ is a monotonically non-increasing and strictly positive
sequence of weights tending to zero. Therefore, the fully data-driven estimator is called adaptive.

\subsection{Non-parametric  regression}\label{s:id:r}
In this paragraph we suppose that the identically distributed $n$-sample $\{(\rOb_i,\rRe_i)\}_{i=1}^n$
consists of independent  random variables.
\begin{prop}\label{id:r:p:ub} Let $\{(\rOb_i,\rRe_i)\}_{i=1}^n$ be an iid. $n-$sample. Under the  assumption
\ref{mo:no:as:ba:i} holds 
\begin{equation}\label{id:r:p:ub:e1}
\Rif{\hDiSo[\oDi]}{\Socwr}\leq (\maxnormsup^2 (\sigma^2+\Sor^2)+\Sor^2)\;\oRa,\quad\mbox{for all }n\geq 1,
\end{equation}
\end{prop}
\begin{prop}\label{id:r:p:lb}
 Suppose $\{(\rOb_i,\rRe_i)\}_{i=1}^n$ is an iid. $n-$sample.  Let the error term be normally distributed and assume further
  that \begin{equation}\label{id:r:p:lb:e1}
    0<\eta:=\inf_{n\geq1}\{(\oRa)^{-1}\min(\Sow_{\oDi}, \oDi
    n^{-1})\leq1,
  \end{equation}
  then for all $n\geq 1$ we have
  \begin{equation}\label{id:r:p:lb:e}
    \inf_{\tSo}\Rif{\tSo}{\Socwr} \geq \frac{\eta}{8}\, \min(2\Sor^2,
    \sigma^2)\, \oRa
  \end{equation}
  where the infimum is to be taken over all possible estimators $\tSo$
  of $\So$.
\end{prop}

Again in the configurations considered in the Illustration \ref{illus} the condition \eqref{id:r:p:lb:e1} hold true. Combining the upper bound \eqref{id:r:p:ub:e1} and the lower bound
\eqref{id:r:p:lb:e} we have shown that $\oRa$ is the minimax-optimal $\Rif{\hDiSo[\widetilde{m}]}{\Socwr}$ by apply the Proposition \ref{an:pr:dd}.
rate of convergence and the estimator $\hDiSo[\oDi]$ is
minimax-optimal.

\paragraph{Partially data-driven estimator.} In this paragraph, we select the dimension parameter  following the procedure sketched in \eqref{mo:me:de:wtm} where the subsequence of
non-negative and non-decreasing penalties
$\left(\pen[1],\dotsc,\pen[n]\right)$ is given by
$\pen=144\sigma_Y^2 \maxnormsup^2 mn^{-1}$ with $\sigma_Y^2=\Ex Y^2$. Since $\sigma_Y$ has to be estimated from the data, the considered selection method leads to a partially data-driven estimator of the non-parametric regression function $\So$ only. In order to apply the Proposition \ref{an:pr:dd} it remains to check the conditions \ref{mo:me:as:est:A1} and \ref{mo:me:as:est:A2}. Keeping in mind the definition of the penalties subsequence, the condition \ref{mo:me:as:est:A1} is obviously satisfied. The next Proposition provides our  key argument to  verify the condition \ref{mo:me:as:est:A2}.
\begin{prop}\label{id:r:p:co}Let $\{(\rOb_i,\rRe_i)\}_{i=1}^n$ be an iid. $n-$sample. Suppose that the assumptions \ref{mo:no:as:ba:i} and
  \ref{mo:no:as:ba:ii} are satisfied. If  $\Ex\epsilon^6<\infty$ then
  there exists a finite constant $C(\Sor\gA,\rNoL,\maxnormsup,\Ex\epsilon^6)$ depending only on the
  quantities $\Sor\gA$, $\rNoL$, $\maxnormsup$ and $\Ex\epsilon^6$ such that
  \begin{equation*}
   \sup_{\So\in\Socwr} \Ex\set{\max_{1\leq \Di\leq n}\vectp{\HnormV{\hDiSo-\DiSo}^2
        -12\maxnormsup^2\sigma_Y^2\Di n^{-1}}}\leq
    n^{-1}C(\Sor\gA,\rNoL,\maxnormsup,\Ex\rNo^{6}),\quad\mbox{for all }n\geq1.
  \end{equation*}
\end{prop}

Obviously, taking into account the definition of penalties sequence the last Proposition shows that the condition
\ref{mo:me:as:est:A2} is satisfied.  Thereby, the
next assertion is an immediate consequence of  Proposition \ref{an:pr:dd}
and we omit its proof.

\begin{prop}\label{id:r:t:pad} Suppose $\{(\rOb_i,\rRe_i)\}_{i=1}^n$ is an iid. $n-$sample. Let assumptions \ref{mo:no:as:ba:i} and
  \ref{mo:no:as:ba:ii} be satisfied. Select the dimension parameter $\tDi$ as given by \eqref{mo:me:de:wtm} with $\pen:=72\maxnormsup^2\sigma_Y^2\Di n^{-1} $. If  $\Ex\epsilon^6<\infty$ then
  there exists a numerical constant $C$ and a finite constant $\zeta(\Sor\gA,\rNoL,\maxnormsup,\Ex\epsilon^6)$ depending only on the  quantities $\Sor\gA$, $\rNoL$, $\maxnormsup$ and $\Ex\epsilon^6$ such that for all $n\geq n_{\aSy}$
  with $\oRa[n_{\aSy}]\leq 1$ we have
  \begin{eqnarray*}
    \Rif{\hSo_{\tDi}}{\Socwr}\leq C[\Sor^2\vee
    \maxnormsup^2\sigma_Y^2\vee
    \zeta(\Sor\gA,\rNoL,\maxnormsup,\Ex\rNo^{6})]\; \oRa.
  \end{eqnarray*}
\end{prop}
Since $\sigma_Y^2=\Ex Y^2$ is generally unknown, the
penalty term specified in the last assertion is not feasible. However,
we have a natural estimator $\widehat{\sigma}_Y^2=n^{-1}\sum_{i=1}^nY_i^2$ of the quantity $\sigma_Y^2$ at hand. 

\paragraph{Fully  data-driven estimator.} In the sequel we consider the  subsequence of
non-negative and non-decreasing penalties
$\left(\hpen[1],\dotsc,\hpen[n]\right)$ given by
$\hpen=144\widehat{\sigma}_Y^2 \maxnormsup^2 mn^{-1}$. The  dimension parameter
$\hDi$ is then selected  as in \eqref{mo:me:de:whm}. Keeping in mind
the Proposition \ref{an:pr:dd:hpen} it remains to show that the
condition \ref{mo:me:as:est:A3} holds true. Therefore, define further the event
$\cV:=\set{{1}/{2}\leq{\widehat{\sigma}_Y^2}/{\sigma_Y^2}\leq{3}/{2}}$
and denote by $\cV^c$ its complement. 
\begin{lem}\label{id:r:l:re}Let $\{(\rOb_i,\rRe_i)\}_{i=1}^n$ be an iid. $n$-sample. If $\Ex\epsilon^4<\infty$, then
  $\sup_{\So\in\Socwr}P(\cV^c)\leq 128 n^{-1}\big((\Ex\epsilon^4)^{1/4}+\Sor\gA/\sigma\big)^4$.
\end{lem}
Considering the event $\Omega$
given in \eqref{mo:me:le:ms:omega} it is easily seen that
$\cV\subset\Omega$ and hence, by employing the last assertion together with Proposition
\ref{id:r:p:co} the conditions  \ref{mo:me:as:est:A1}-\ref{mo:me:as:est:A3} are satisfied. Thereby, the
next assertion is an immediate consequence of  Proposition \ref{an:pr:dd:hpen}
and we omit its proof.

\begin{theo}\label{id:r:t:fad} Suppose $\{(\rOb_i,\rRe_i)\}_{i=1}^n$ is an iid. $n-$sample. Let assumptions \ref{mo:no:as:ba:i} and
  \ref{mo:no:as:ba:ii} be satisfied. Select the dimension parameter
  $\hDi$ as given by \eqref{mo:me:de:whm} with  $\hpen:= 144 \maxnormsup^2\hsigma_Y^2\Di n^{-1}$. If  $\Ex\epsilon^6<\infty$ then
  there exists a numerical constant $C$ and a finite constant $\zeta(\Sor\gA,\rNoL,\maxnormsup,\Ex\epsilon^6)$ depending only on the  quantities $\Sor\gA$, $\rNoL$, $\maxnormsup$ and $\Ex\epsilon^6$ such that for all $n\geq n_{\aSy}$
  with $\oRa[n_{\aSy}]\leq 1$ we have
  \begin{eqnarray*}
    \Rif{\hDiSo[\hDi]}{\Socwr}\leq C[\Sor^2\vee
    \maxnormsup^2\sigma_Y^2\vee
    \zeta(\Sor\gA,\rNoL,\maxnormsup,\Ex\rNo^{6})]\; \oRa.
  \end{eqnarray*}
\end{theo}
We shall emphasise that the last assertion establishes the
minimax-optimality of the  fully data-driven estimator $\hDiSo[\hDi]$
over all classes $\Socwr$. Therefore, the estimator is called adaptive.

%%% Local Variables: 
%%% mode: latex
%%% TeX-master: "_0DP_NPE_dep"
%%% End: 

%======================================================================================================================
%                                                                 
% Title: Dependent observations
% Author: Jan JOHANNES, Nicolas Asin
% 
% Email: jan.johannes@ensai.fr nicolas.asin@uclouvain.be
% Date: %%ts latex start%%[2016-02-01 Mon 02:16]%%ts latex end%%
% Main-TeX-File: t in der form "Name"
%
% ======================================================================================================================
\section{Dependent observations}\label{s:dd}

% --------------------------------------------------------------------
% <<Coupling \label{s:dd:mi}>>
% --------------------------------------------------------------------
In this section we dismiss the
independence assumption and assume weakly dependent
observations. More precisely, $\Ob_1,\dotsc,\Ob_n$ are drawn from a
strictly stationary process $(\Ob_i)_{i\in\Zz}$ taking still its values
in $[0,1]$. Keep in mind that a process is called strictly stationary
if its finite dimensional distributions does not change when shifted
in time. Consequently, the random variables $\{\Ob_i\}$ are identically distributed.  Our aim is the non-parametric estimation of the function $\So$ under some
mixing conditions on the dependence of the process $(\Ob_i)_{i\in\Zz}$. Let us begin with 
a brief review of a classical measure of dependence, leading to the
notion of a stationary absolutely regular process. 

Let $(\Omega,\sA,P)$ be a probability space. Given two
$\sigma$-algebras $\sU$ and $\sV$ of $\sA$ we introduce next
the definition and properties of the absolutely regular mixing (or
$\beta$-mixing) coefficient $\beta(\sU,\sV)$.    The coefficient was
introduced by \cite{KolmogorovRozanov1960} and is defined by 
\begin{equation*}
  \beta(\sU,\sV)
  :=\tfrac{1}{2}\sup\set{\sum_{i}\sum_j\left|P(U_i)P(V_i)-P(U_i\cap V_i)\right|}
\end{equation*}
where the supremum is taken over all finite partitions $(U_i)_{i\in
  I}$ and $(V_j)_{j\in J}$, which are respectively $\sU$ and $\sV$
measurable. Obviously, $\beta(\sU,\sV)\leq 1$. As usual, if $\Ob$ and $\Ob^\prime$ are two real-valued random variables, we denote by $\beta(\Ob,\Ob^\prime)$ the mixing coefficient
$\beta(\sigma(\Ob),\sigma(\Ob^\prime))$, where  $\sigma(\Ob)$ and $\sigma(\Ob^\prime)$ are,
respectively, the $\sigma$-fields generated by $\Ob$ and $\Ob^\prime$.
Consider a strictly stationary process $(\Ob_i)_{i\in\Zz}$ then for any
integer $k$ the mixing coefficient $\beta(\Ob_0,\Ob_k)$ does not change
when shifted over time, i.e., $\beta(\Ob_0,\Ob_k)=\beta(\Ob_{0+l},\Ob_{k+l})$
for all integer $l$. The next assertion follows along
the lines of the proof of  Theorem 2.1 in \cite{Viennet1997} and we omit its proof.
\begin{lem}\label{dd:le:var}
Let $(\Ob_i)_{i\in\Zz}$ be a strictly stationary process of
  real-valued random variables. There exists a sequence $(b_k)_{k\geq1}$ of measurable
  functions $b_k:\Rz\to[0,1]$ with  $\Ex b_k(\Ob_0)=\beta(\Ob_0,\Ob_k)$ such that for any measurable function $h$
  with $\Ex |h(\Ob_0)|^2<\infty$ and any integer $n$,
  \begin{equation*}
   \Var(\sum_{i=1}^nh(\Ob_i))\leq n\, \Ex\bigg\{|h(\Ob_0)|^2\big(1+4\sum_{k=1}^{n-1}b_k(\Ob_0)\big)\bigg\}.
  \end{equation*}
\end{lem}

Given  $p\geq 2$, a non-negative sequence $w:=(w_k)_{k\geq 0}$ and a
probability measure $P$ let $\sL(p,w,P)$ be the set of functions $b:\Rz\to [0,\infty]$ such
that there exists a sequence $(b_k)_{k\geq0}$ of measurable
  functions $b_k:\Rz\to[0,1]$, with $b_0=\1$ and
  $\Ex_P b_k\leq w_k$ satisfying
  $b=\sum_{k=0}^\infty(k+1)^{p-2}b_k$. We note that the elements of $\sL(p,w,P)$ are generally
  not $P$-integrable, however,  whenever
  $\sum_{k=0}^\infty(k+1)^{p-2}w_k<\infty$, each function $b$ in
  $\sL(p,w,P)$ is a non-negative $P$-integrable function. Moreover,
  reconsidering a strictly stationary process $(\Ob_i)_{i\in\Zz}$ with
  common marginal distribution $P_{\Ob_0}$ and associated non-negative
  sequence  of mixing coefficients $w=(w_k)_{k\geq0}$ with $w_0=1$ and $w_k=\beta(\Ob_0,\Ob_k)$ 
  an immediate consequence of  Lemma  \ref{dd:le:var} is the
  existence of a function $b$ belonging to $\sL(2,\beta,P_{\Ob_0})$ such
  that for any measurable function $h$
  with $\Ex |h(\Ob_0)|^2<\infty$ and any integer $n$,
  \begin{equation}\label{dd:le:var:e3}
\Var(\sum_{i=1}^nh(\Ob_i))\leq 4n \Ex (|h(\Ob_0)|^2b(\Ob_0)).
\end{equation}
Note that the assumptions stated yet do not ensure that the right hand
side in the last display is finite. However, the function $b$ is
$P_{\Ob_0}$-integrable whenever
$\sum_{k\geq1}\beta(\Ob_0,\Ob_k)<\infty$. Therefore, imposing in addition that
$\sum_{k\geq1}\beta(\Ob_0,\Ob_k)<\infty$ and, for example, that $\normInf{h}<\infty$ we have  $\Ex
(h(\Ob_0)|^2b(\Ob_0)|\leq  \normInf{h}\Ex b(\Ob_0)<\infty$.  Obviously, given conjugate exponents $p$ and $q$ if
$b$ has a finite $p$-th moment, i.e., $\Ex |b(\Ob_0)|^p<\infty$, and
$\Ex|h(\Ob_0)|^{2q}<\infty$, then we have  $\Ex(h(\Ob_0)|^2b(\Ob_0)|\leq  \{\Ex |h(\Ob_0)|^{2q}\}^{1/q}\{\Ex |b(\Ob_0)|^p\}^{1/p}<\infty$.
Lemma 4.2 in \cite{Viennet1997} provides now sufficient conditions to ensure the
existence of a finite $p$-th moment of $b$ which is summarized in the next assertion.
  \begin{lem}\label{dd:le:b}  Let the sequence $w:=(w_k)_{k\geq 0}$ be non-increasing,
    tending to $0$ as $k\to\infty$ with $w_0=1$ and such that $\sum_{k=0}^\infty(k+1)^{p-1}w_k<\infty$ for
    some $1\leq p\leq\infty$. Then, for each $b$ in $\sL(2,w,P)$ the
    function $b^p$ is $P$-integrable and $\Ex_P|b|^p\leq p \sum_{k=0}^\infty(k+1)^{p-1}w_k$.
  \end{lem}
 
We will use Lemma \ref{dd:le:var}, the estimate \eqref{dd:le:var:e3} together with Lemma \ref{dd:le:b} to derive an upper bound for the maximal risk of the non-parametric estimator with suitable choice of the dimension parameter. However, in order to control the deviation of the data-driven estimator, more precisely in order to show that the condition \ref{mo:me:as:est:A2} holds true, we have made use of Talagrand's inequality which is formulated for independent observations only. Inspired by the work of \cite{ComteDedeckerTaupin2008} we will use coupling techniques to extend Talagrand's inequality to dependent data which we present next.
    We assume in the sequel that there exists a sequence 
of independent random variables with  uniform distribution on $[0,1]$
independent of the sequence $(\Ob_i)_{i\geq1}$. Employing Lemma 5.1 in \cite{Viennet1997} we construct by induction a sequence
$(\cou{\Ob}_i)_{i\geq1}$ satisfying the following properties.  Given an
integer $q$ we  introduce  disjoint even and odd blocks of indices, i.e.,
for any $l\geq1$, $\cI^e_{l}:=\{2(l-1)q+1,\dotsc,(2l-1)q\}$ and
$\cI^o_{l}:=\{(2l-1)q+1,\dotsc,2lq\}$, respectively, of size
$q$. Let us further partition into blocks
the random processes  $(\Ob_i)_{i\geq 1}=(E_l,O_l)_{l\geq1}$ and
$(\cou{\Ob}_i)_{i\geq1}=(\couE_l,\couO_l)_{l\geq1}$ where
\begin{align*}
	E_l=(\Ob_i)_{i\in\cI^e_{l}},\qquad
        \couE_l=(\cou{\Ob}_i)_{i\in\cI^e_{l}},\qquad
        O_l=(\Ob_i)_{i\in\cI^o_{l}}, \qquad \couO_l=(\cou{\Ob}_i)_{i\in\cI^o_{l}}.
\end{align*}
If we set further $\sF_l^-:=\sigma(\Ob_j,j\leq l)$ and $\sF_l^+:=\sigma(\Ob_j,j\geq l)$, then the sequence $(\beta_k)_{k\geq 0}$ of $\beta$-mixing coefficient defined by $\beta_0:=1$ and $\beta_k:=\beta(\sF_0^-,\sF_k^+)$, $k\geq 1$, is monotonically non-increasing and satisfies trivially $\beta_k\geq\beta(\Ob_0,\Ob_k)$  for any $k\geq1$.
\noindent Based on the construction presented in \cite{Viennet1997}, the sequence $(\cou{\Ob}_i)_{i\geq1}$ can be
chosen such that for any integer $l\geq1$: 
\begin{enumerate}[label={\textbf{(P\arabic*)}},ref={\textbf{(P\arabic*)}}]\addtocounter{enumi}{0}
\item\label{dd:as:cou1} $\couE_l$, $E_l$, $\couO_l$ and $O_l$ are identically distributed,
\item\label{dd:as:cou2} $P(E_l\ne \couE_l)\leq \beta_{q+1}$, and $P(O_l\ne \couO_l)\leq \beta_{q+1}$.
\item\label{dd:as:cou3} The variables $(\couE_1,\dotsc,\couE_l)$ are iid. and so $(\couO_1,\dotsc,\couO_l)$.
\end{enumerate}

We may emphasise that the random vectors $\couE_1,\dotsc,\couE_l$ are
iid. but the components within each vector are generally not independent.

% --------------------------------------------------------------------
% <<Density \label{s:dd:d}>>
% --------------------------------------------------------------------
\subsection{Non-parametric  density estimation}\label{s:dd:d}
Let us turn our attention back to the orthogonal series
estimator defined in the paragraph \ref{s:mo:ob}. Keep in mind  that
$\dOb_1,\dotsc,\dOb_n$ are drawn from a strictly stationary process
$(\dOb_i)_{i\in\Zz}$ with common marginal distribution admitting a
density $\So$. Exploiting the assumption
\ref{mo:no:as:ba:i} and Lemma \ref{dd:le:var} we obtain the next assertion
\begin{prop}[Upper bound]\label{dd:pr:ub}Let $(\dOb_i)_{i\in\Zz}$ be  a strictly stationary process with associated sequence of mixing coefficients $\set{\beta(\dOb_0,\dOb_k)}_{k\geq1}$. Under assumption
\ref{mo:no:as:ba:i} holds 
\begin{equation}\label{dd:pr:ub:e1}
\Rif{\hDiSo[\oDi]}{\Socwr\cap\cD}\leq (\maxnormsup^2\{1+4\sum_{k=1}^{n-1}\beta(\dOb_0,\dOb_k)\}+\Sor^2)\;\oRa,\quad\mbox{for all }n\geq 1.
\end{equation}
\end{prop}
Let us compare briefly the last result and  the  upper risk bound  assuming independent
observations given in Proposition
\ref{id:d:p:ub}. We see, that this  upper risk bound provides up to finite constant also an upper risk bound
in the presence of dependence whenever
$\sum_{k=1}^{\infty}\beta(\dOb_0,\dOb_k)<\infty$. 
However, the upper bound given in Proposition
\ref{dd:pr:ub} depends on the unknown mixing coefficients
$\{\beta(\dOb_0,\dOb_k)\}_k$. Their estimation is a demanding task, and
hence, we next derive an upper bound which does not
depend on the mixing coefficients at least for all sufficiently
large sample sizes $n$. This upper bound relies on the next assumption which has been used, for example, in \cite{Bosq1998}.
\begin{enumerate}[label={\textbf{(D\arabic*)}},ref={\textbf{(D\arabic*)}}]\addtocounter{enumi}{0}
\item\label{dd:as:de:i}
For any integer $k$ the joint distribution $P_{\dOb_0,\dOb_k}$ of
$(\dOb_0,\dOb_k)$ admits a density $f_{\dOb_0,\dOb_k}$ which is square
integrable. Let
$\normV{f_{\dOb_0,\dOb_k}}^2:=\int_0^1\int_0^1|f_{\dOb_0,\dOb_k}(x,y)|^2dxdy<\infty$
with a slight abuse of notations. If we denote further by $h\otimes
g:[0,1]^2\to\Rz$ the bivariate function $[h\otimes g](x,y):=h(x)g(y)$,
then let
$\gamma_{f}:=\sup_{k\geq1}\normV{f_{\dOb_0,\dOb_k}-f\otimes f}<\infty$.
\end{enumerate}
\begin{lem}\label{dd:le:est:var}Let $(\dOb_i)_{i\in\Zz}$ be  a strictly stationary process with associated sequence of mixing coefficients $\set{\beta(\dOb_0,\dOb_k)}_{k\geq1}$. Under the assumptions
\ref{mo:no:as:ba:i} and \ref{dd:as:de:i} for any  $n\geq 1$ and $K\in\{0,\dotsc,n-1\}$ it holds
\begin{equation}\label{dd:le:est:var:e1}
 \sum_{j=1}^\Di  \Var(\sum_{i=1}^n\bas_j(\dOb_i))\leq  n m \{\maxnormsup^2 + 2 [ \gamma_f K/\sqrt{m} + 2
\maxnormsup^2 \sum_{k=K+1}^{n-1}\beta(\dOb_0,\dOb_k)]\}.
\end{equation}
\end{lem}
If we assume in addition that
$\sum_{k=1}^{\infty}\beta(\dOb_0,\dOb_k)<\infty$ and $\gamma:=\sup_{f\in\Socwr\cap\cD}\gamma_f<\infty$  then there exist an
integer $K_o$ and an integer $n_o$  such that $\sum_{k=K_o+1}^{\infty}\beta(\dOb_0,\dOb_k)<1/8$
and $K_n:=\gauss{ 4  \maxnormsup^2 \sqrt{\oDi} /\gamma}\geq K_o$ with $\oDi$ as given in \eqref{mo:me:de:ra} for all $n\geq n_o$. Thereby,
 we have for all  $n\geq n_o$ that $ \sum_{j=1}^{\oDi}
\Var(\sum_{i=1}^n\bas_j(\dOb_i))\leq \maxnormsup^2 n \,\oDi $. We note that
$n_o$ depends on the sequence of mixing coefficients. The next
assertion is an immediate consequence and we omit its proof.
\begin{prop}[Upper bound]\label{dd:pr:ub2}Let $(\dOb_i)_{i\in\Zz}$ be  a strictly stationary process with associated sequence of mixing coefficients $\set{\beta(\dOb_0,\dOb_k)}_{k\geq1}$. Under Condition
\ref{mo:no:as:ba:i} and  \ref{dd:as:de:i} if 
$\sum_{k=1}^{\infty}\beta(\dOb_0,\dOb_k)<\infty$ and $\gamma:=\sup_{f\in\Socwr\cap\cD}\gamma_f<\infty$ then there exists an
integer  $n_o$ (possibly depending on the mixing coefficients and $\gamma$)
such that 
\begin{equation}\label{dd:pr:ub2:e1}
\Rif{\hDiSo[\oDi]}{\Socwr\cap\cD}\leq (\maxnormsup^2+\Sor^2)\;\oRa,\quad\mbox{for all }n\geq n_o.
\end{equation}
\end{prop}
Consequently under the condition of Proposition \ref{dd:pr:ub2} the estimator $\hDiSo[\oDi]$ attains the minimax-optimal rate  $\oRa$ for independent data 

\paragraph{Fully data-driven estimator.} 
Consider the estimator $\hDiSo[\widetilde{m}]$ where $\widetilde{m}$ is defined in \eqref{mo:me:de:wtm} with $\pen:= 288 \maxnormsup^2\Di n^{-1}$. We aim to derive  an upper bound for its maximal risk   $\Rif{\hDiSo[\widetilde{m}]}{\Socwr\cap\cD}$ by making use of Proposition \ref{an:pr:dd}. Therefore, it remains to check the conditions \ref{mo:me:as:est:A1} and \ref{mo:me:as:est:A2} where \ref{mo:me:as:est:A1} holds obviously true due to the definition of  penalty term. The next assertion provides our key argument in order to verify the condition \ref{mo:me:as:est:A2}.

\begin{prop}\label{dd:co:dd:as1} Let $(\dOb_i)_{i\in\Zz}$ be  a strictly stationary process with associated sequence of mixing coefficients $(\beta_k)_{k\geq1}$ satisfying $\gB:= 2\sum_{k=0}^\infty(k+1)\beta_k<\infty$. Under the assumptions \ref{mo:no:as:ba:i}, \ref{mo:no:as:ba:ii} and \ref{dd:as:de:i}, let $\gamma:=\sup_{f\in\Socwr\cap\cD}\gamma_f<\infty$, $K_n:=\gauss{ 4  \maxnormsup^2 \sqrt{\oDi} /\gamma}$  and  $\mu_n\geq  \{3+ 8 \sum_{k=K_n+1}^{\infty}\beta_k\}$. There
  exists a numerical constant $C>0$ such that for any integer $q$
\begin{multline}\label{dd:co:dd:as1:eq}
\sup_{\So\in\Socwr}\Ex\set{\max_{\oDi\leq \Di\leq n}\vectp{\normV{\hDiSo-\DiSo}^2
    -12\maxnormsup^2\Di n^{-1}\mu_n}}\\\leq C\; n^{-1} \maxnormsup^2\bigg\{\mu_n
        \Psi\bigg(\frac{\Sor\gA\gB}{\maxnormsup^2
          \mu_n^2}\bigg)+  n
        q^2\exp\left(-\frac{n^{1/2}}{q}\frac{\mu_n^{1/2}}{144}\right)
+  n^2 \beta_{q+1}\bigg\}
\end{multline}
where $\Psi(x):=\sum_{m\geq 1}^\infty x^{1/2} m^{1/2}\exp(-
m^{1/2}/(48x^{1/2}))<\infty$, for any $x>0$.
\end{prop}
Note that the condition $\gB= 2\sum_{k=0}^\infty(k+1)\beta_k<\infty$
implies  $\sum_{k=K_n+1}^{\infty}\beta_k\leq (K_n+1)^{-1}\gB$ and hence,  $\{3+ 8
\sum_{k=K_n+1}^{\infty}\beta_k\}\leq 4$ whenever $K_n=\gauss{ 4  \maxnormsup^2 \sqrt{\oDi} /\gamma}\geq 8\gB$. Since $\oDi\to
 \infty$ as $n\to\infty$ there exists an integer $n_o$ such that
for all $n\geq n_o$ we can chose $\mu_n=4$. The next assertion is thus
an immediate consequence of Proposition \ref{dd:co:dd:as1}, and hence we
omit its proof. 
\begin{coro}\label{dd:pr:dd:as1}  Let the assumptions of Proposition  \ref{dd:co:dd:as1} be satisfied. Suppose that
  there exists an unbounded sequence of integers
$(q_n)_{n\geq1}$ and a finite constant $L>0$  such that 
\begin{equation}\label{dd:pr:dd:as1:e1}
        \sup_{n\geq
          1}nq_n^2\exp\left(-\frac{n^{1/2}}{q_n}\frac{1}{72}\right)\leq
        L\quad\mbox{ and }\quad\sup_{n\geq1} n^2 \beta_{q_n+1}\leq L.
      \end{equation} There
  exist a numerical constant $C>0$ and  an integer $n_o$ such that
for all $n\geq n_o$
\begin{equation*}
\sup_{\So\in\Socwr}\Ex\set{\max_{\oDi\leq \Di\leq n}\vectp{\normV{\hDiSo-\DiSo}^2
    -48\maxnormsup^2\Di n^{-1}}}\leq C  n^{-1} \maxnormsup^2\bigg\{\Psi\bigg(\frac{\Sor\gA\gB}{16\maxnormsup^2}\bigg)+ L\bigg\}.
\end{equation*}
\end{coro}
Is it interesting to note that an arithmetically decaying sequence
of mixing coefficients $(\beta_k)_{k\geq1}$ satisfies \eqref{dd:pr:dd:as1:e1}.
To be more precise, consider two sequence of integers $(q_n)_{n\geq1}$, $(p_n)_{n\geq1}$ such that $n=2q_np_n$ and assume %Let us briefly comment the additional  condition \eqref{dd:pr:dd:as1:e1}. 
additionally $\beta_k\leq
k^{-s}$. The sequence $q_n\asymp n^{p_n}$, i.e., $(n^{-p_n}q_n)_{n\geq1}$
is bounded away both from zero and infinity, and satisfies the condition
\eqref{dd:pr:dd:as1:e1} whenever $2<p_ns$ and $1/2>p_n$. In other words,
if the sequence  of mixing coefficients $(\beta_k)_{k\geq1}$ is
sufficiently fast decaying, that is $s > 2 (2 +\theta)$ for some
$\theta>0$, then the  condition \eqref{dd:pr:dd:as1:e1} holds true
taking, for example, a sequence  $q_n\asymp n^{1/(2+\theta)}$.

Obviously, using the penalty $\pen:= 288 \maxnormsup^2\Di n^{-1}$ for any $m\in\Nz$ the conditions \ref{mo:me:as:est:A1} and
\ref{mo:me:as:est:A2} due to  Proposition 
\ref{dd:pr:dd:as1} are satisfied.
Thereby, the
next assertion is an immediate consequence of  Proposition \ref{an:pr:dd}
and we omit its proof.
\begin{theo}\label{dd:th:dd} Under the assumptions of Proposition \ref{dd:co:dd:as1} and the condition \eqref{dd:pr:dd:as1:e1}
there
  exist a numerical constant $C>0$ and an integer $n_o$ such that
for all $n\geq n_o$ we have
\begin{equation*}
\Rif{\hDiSo[\whm]}{\Socwr\cap\cD}\leq C \,\bigg[\Sor\vee\maxnormsup^2\vee \maxnormsup^2\big\{\Psi\big(\frac{\Sor\gA\gB}{16\maxnormsup^2}\big)+
L\big\}\bigg]\;\oRa.
\end{equation*}
\end{theo}
Note that the
penalty term depends only on known quantities and, hence the
$\hSo_{\whm}$ is fully data-driven.
The last assertion establishes the minimax-rate optimality of the
fully data-driven estimator $\hSo_{\whm}$ over all classes
$\Socwr\cap\cD$. Therefore, the estimator is called adaptive. 

% --------------------------------------------------------------------
% <<Regressiom \label{s:dd:r}>>
% --------------------------------------------------------------------
\subsection{Non-parametric regression}\label{s:dd:r}
Let us turn our attention to the orthogonal series
estimator defined in the paragraph \ref{s:mo:ob}. In the sequel we
suppose  that the explanatory variables
$\rRe_1,\dotsc,\rRe_n$ are drawn from a strictly stationary process
$(\rRe_i)_{i\in\Zz}$ with common marginal uniform distribution on the
interval $[0,1]$. Moreover, we still assume that the error terms
$\{\rNo_i\}_{i=1}^n$ are iid. and independent to the explanatory variables. Exploiting the assumption
\ref{mo:no:as:ba:i} and Lemma \ref{dd:le:var} we obtain  the next assertion 
\begin{prop}[Upper bound]\label{dd:r:p:ub}Let $(\rRe_i)_{i\in\Zz}$ be  a strictly stationary process with associated sequence of mixing coefficients $\set{\beta(\rRe_0,\rRe_k)}_{k\geq1}$. Under
\ref{mo:no:as:ba:i} holds 
\begin{equation}\label{dd:r:p:ub:e1}
\Rif{\hDiSo[\oDi]}{\Socwr}\leq (\rNoL^2 +\inormV{\So}^2\maxnormsup^2\{1+4\sum_{k=1}^{n-1}\beta(\rRe_0,\rRe_k)\}+\Sor^2)\;\oRa,\quad\mbox{for all }n\geq 1.
\end{equation}
\end{prop}
Comparing the last result and  Proposition
\ref{id:r:p:ub} the  upper risk bound  assuming independent
observations provides  up to a finite constant also an  upper risk bound
in the presence of dependence whenever
$\sum_{k=1}^{\infty}\beta(\rRe_0,\rRe_k)<\infty$. 

\begin{enumerate}[label={\textbf{(D\arabic*)}},ref={\textbf{(D\arabic*)}}]\addtocounter{enumi}{1}
\item\label{dd:as:re:i}
For any integer $k$ the joint distribution $P_{\rRe_0,\rRe_k}$ of
$(\rRe_0,\rRe_k)$ admits a density $f_{\rRe_0,\rRe_k}$ which is square
integrable and satifies
$\gamma:=\sup_{k\geq1}\normV{f_{\rRe_0,\rRe_k}-\1\otimes\1}<\infty$.
\end{enumerate}
\begin{lem}\label{dd:r:l:est:var}Let $(\rRe_i)_{i\in\Zz}$ be  a strictly stationary process with associated sequence of mixing coefficients $\set{\beta(\rRe_0,\rRe_k)}_{k\geq1}$. Under assumptions
\ref{mo:no:as:ba:i} and \ref{dd:as:re:i} holds for any  $n\geq 1$ and $K\in\{0,\dotsc,n-1\}$
\begin{equation}\label{dd:r:l:est:var:e1}
 \sum_{j=1}^\Di  \Var(\sum_{i=1}^n\So(\rRe_i)\bas_j(\rRe_i))\leq  n m \{\maxnormsup^2\HnormV{\So}^2 + 2\inormV{\So}^2 [ \gamma K/\sqrt{m} + 2
\maxnormsup^2 \sum_{k=K+1}^{n-1}\beta(\rRe_0,\rRe_k)]\}.
\end{equation}
\end{lem}
Note that supposing further assumption
\ref{mo:no:as:ba:ii} we have  $\inormV{\So}^2\leq \Sor^2\gA^2$  for all $\So\in\Socwr$.
If we assume in addition that
$\sum_{k=1}^{\infty}\beta(\rRe_0,\rRe_k)<\infty$ then there exists an
integer $K_o$ and an integer $n_o$  such that $\sum_{k=K_o+1}^{\infty}\beta(\rRe_0,\rRe_k)<1/(8\Sor^2\gA^2)$
and $K_n:=\gauss{\maxnormsup^2 \sqrt{\oDi} /(\gamma\Sor^2\gA^2)}\geq K_o$ for all $n\geq n_o$. Thereby,
 we have for all  $n\geq n_o$ that $ \sum_{j=1}^{\oDi}
\Var(\sum_{i=1}^n\So(\rRe_i)\bas_j(\rRe_i))\leq (\Sor^2
+1)\maxnormsup^2 n \,\oDi $ for all $\So\in\Socwr$. We note that
$n_o$ depends on the sequence of mixing coefficients and the
quantity $\Sor\gA$. The next
assertion is an immediate consequence and we omit its proof.
\begin{prop}[Upper bound]\label{dd:r:p:ub2}Let $(\rRe_i)_{i\in\Zz}$ be  a strictly stationary process with associated sequence of mixing coefficients $\set{\beta(\rRe_0,\rRe_k)}_{k\geq1}$. Let assumptions
\ref{mo:no:as:ba:i}, \ref{mo:no:as:ba:ii}, \ref{dd:as:re:i} and 
$\sum_{k=1}^{\infty}\beta(\rRe_0,\rRe_k)<\infty$ be satisfied. There exists an
integer  $n_o$ (possibly depending on the mixing coefficients and the
quantity $\Sor\gA$)
such that 
\begin{equation}\label{dd:r:p:ub2:e1}
\Rif{\hDiSo[\oDi]}{\Socwr}\leq (\rNoL^2+(\Sor^2 +1)\maxnormsup^2+\Sor^2)\;\oRa,\quad\mbox{for all }n\geq n_o.
\end{equation}
\end{prop}

\paragraph{Partially data-driven  estimator.} In this paragraph, we select the dimension parameter  following the procedure sketched in \eqref{mo:me:de:wtm} where the subsequence of
non-negative and non-decreasing penalties
$\left(\pen[1],\dotsc,\pen[n]\right)$ is given by
$\pen=1152\sigma_Y^2 \maxnormsup^2 mn^{-1}$ with $\sigma_Y^2=\Ex Y^2$. Since $\sigma_Y$ has to be estimated from the data, the considered selection method leads to a partially data-driven estimator of the non-parametric regression function $\So$ only. In order to apply the Proposition \ref{an:pr:dd} it remains to check the conditions \ref{mo:me:as:est:A1} and \ref{mo:me:as:est:A2}. Keeping in mind the definition of the penalties subsequence, the condition \ref{mo:me:as:est:A1} is obviously satisfied. The next Proposition provides our  key argument to  verify the condition \ref{mo:me:as:est:A2}.
 \begin{prop}\label{dd:r:co:dd:as1}Let $(\rRe_i)_{i\in\Zz}$ be  a strictly stationary process with associated sequence of mixing coefficients $(\beta_k)_{k\geq1}$ satisfying $\gB:= 2\sum_{k=0}^\infty(k+1)\beta_k<\infty$.   Under the assumptions of Proposition \ref{dd:r:p:ub2}, let $K_n:=\gauss{ 4  \maxnormsup^2 \HnormV{\So}^2
  \sqrt{\oDi} /(\gamma \Sor^2\gA^2)}$ and $\mu_n\geq3/2+ 4 \sum_{k=K_n+1}^{\infty}\beta_k$. If $\Ex\rNo^6<\infty$, then there exist a finite constant
$\zeta(\Sor\gA,\rNoL,\maxnormsup,\gB,\Ex\rNo^6)$ depending on the quantities
$\Sor\gA$, $\rNoL$, $\maxnormsup$, $\gB$ and $\Ex\rNo^6$ only and  a numerical constant $C>0$ such that for any integer $q$ 
\begin{multline*}
\sup_{\So\in\Socwr}\Ex\set{\max_{\oDi\leq \Di\leq n}\vectp{\normV{\hDiSo-\DiSo}^2
    -24\maxnormsup^2\Di n^{-1}\sigma_Y^2\mu_n}}\\\leq C\; n^{-1}(\rNoL+\Sor\gA)^2\bigg\{\zeta(\Sor\gA,\rNoL,\maxnormsup,\gB,\Ex\rNo^6)\hfill+  n^{3/2} q^2\exp\left(-\frac{n^{1/4}}{q}\frac{1}{576(1+\Sor\gA/\rNoL)}\right)\bigg\}+  n^2 \beta_{q+1}\bigg\}.
\end{multline*}
\end{prop}
Note that the condition $\gB= 2\sum_{k=0}^\infty(k+1)\beta_k<\infty$
implies  $\sum_{k=K_n+1}^{\infty}\beta(\rRe_0,\rRe_k)\leq
\sum_{k=K_n+1}^{\infty}\beta_k\leq (K_n+1)^{-1}\gB$ and hence,  $\{3/2+ 4
\sum_{k=K_n+1}^{\infty}\beta(\rRe_0,\rRe_k)\}\leq 2$ whenever $K_n=\gauss{ \maxnormsup^2 \sqrt{\oDi} /(\gamma\Sor^2\gA^2)}\geq 4\gB$. Since $\oDi\to
 \infty$ as $n\to\infty$ there exists an integer $n_o$ such that
for all $n\geq n_o$ we can chose $\mu_n=2$. The next assertion is thus
an immediate consequence of Corollary \ref{dd:r:co:dd:as1}, and hence we
omit its proof. 
\begin{coro}\label{dd:r:pr:dd:as1} Let the assumptions of Proposition  \ref{dd:r:co:dd:as1} be satisfied. Suppose that
  there exists an unbounded sequence of integers
$(q_n)_{n\geq1}$ and a finite constant $L>0$  such that 
\begin{equation}\label{dd:r:pr:dd:as1:e1}
        \sup_{n\geq
          1}n^{3/2}q_n^2\exp\left(-\frac{n^{1/4}}{q_n}\frac{1}{576(1+\Sor\gA/\rNoL)}\right)\leq
        L\quad\mbox{ and }\quad\sup_{n\geq1} n^2 \beta_{q_n+1}\leq L.
      \end{equation} Then  there
  exist a numerical constant $C>0$ and  an integer $n_o$ such that
for all $n\geq n_o$
\begin{multline*}
\sup_{\So\in\Socwr}\Ex\set{\max_{\oDi\leq \Di\leq n}\vectp{\normV{\hDiSo-\DiSo}^2
    -48\maxnormsup^2\Di n^{-1}\sigma_Y^2}}\\
    \leq C  n^{-1}
(\rNoL+\Sor\gA)^2\big\{\zeta(\Sor\gA,\rNoL,\maxnormsup,\gB,\Ex\rNo^6)+ L\big\}.
\end{multline*}
\end{coro}
Let us briefly comment on the additional  condition
\eqref{dd:r:pr:dd:as1:e1}. Consider two sequence of integers
$(q_n)_{n\geq1}$, $(p_n)_{n\geq1}$ such that $n=2q_np_n$ and assume additionally a polynomial decay of the sequence
of mixing coefficients $(\beta_k)_{k\geq1}$, that is $\beta_k\leq
k^{-s}$. The sequence $q_n\asymp n^{p_n}$, i.e., $(n^{-p_n}q_n)_{n\geq1}$
is bounded away both from zero and infinity, satisfies then the condition
\eqref{dd:r:pr:dd:as1:e1} if $2<p_ns$ and $1/4>p_n$. In other words,
if the sequence  of mixing coefficients $(\beta_k)_{k\geq1}$ is
sufficiently fast decaying, that is $s > 2 (4 +\theta)$ for some
$\theta>0$, then the  condition \eqref{dd:r:pr:dd:as1:e1} holds true
taking a sequence  $q_n\asymp n^{1/(4+\theta)}$.

Obviously taking into account Proposition 
\ref{dd:r:pr:dd:as1} the conditions \ref{mo:me:as:est:A1} and
\ref{mo:me:as:est:A2} are satisfied.
Thereby, the
next assertion is an immediate consequence of  Proposition \ref{an:pr:dd}
and we omit its proof.
\begin{prop}\label{dd:r:th:dd} Under  the assumptions of Proposition \ref{dd:r:pr:dd:as1} and the condition \eqref{dd:r:pr:dd:as1:e1},
there
  exist a numerical constant $C>0$ and exists an integer $n_o$ such that
for all $n\geq n_o$ we have
\begin{equation*}
\Rif{\hDiSo[\widetilde m]}{\Socwr\cap\cD}\leq C \,\big[\Sor^2\vee\maxnormsup^2\vee (\rNoL+\Sor\gA)^2\big\{\zeta(\Sor\gA,\rNoL,\maxnormsup,\gB,\Ex\rNo^6)+ L\big\}\big]\;\oRa.
\end{equation*}
\end{prop}
\paragraph{Fully data-driven  estimator.}
Note that in general $\sigma_Y^2=\Ex Y^2$ is unknown  and hence the
penalty term specified in the last assertion is not feasible, but it can be estimated straightforwardly by $\widehat{\sigma}_Y^2=n^{-1}\sum_{i=1}^nY_i^2$. Consequently, we consider next the sub-sequence of non-negative and non-decreasing penalties $\left(\hpen[1],\dotsc,\hpen[n]\right)$ given by $\hpen=1152 \maxnormsup^2\Di n^{-1}\widehat\sigma_Y^2$.
$\widehat{\sigma}_Y^2=n^{-1}\sum_{i=1}^nY_i^2$ of the quantity $\sigma_Y^2$ at hand. The  dimension parameter
$\hDi$ is then selected  as in \eqref{mo:me:de:whm}. Keeping in mind
the Proposition \ref{an:pr:dd:hpen} it remains to show that the
Condition \ref{mo:me:as:est:A3} holds true. Consider again the event
$\cV:=\set{{1}/{2}\leq{\widehat{\sigma}_Y^2}/{\sigma_Y^2}\leq{3}/{2}}$
and its complement $\cV^c$. 

\begin{lem}\label{dd:r:l:re}Let $(\rRe_i)_{i\in\Zz}$ be  a strictly stationary process with associated sequence of mixing coefficients $(\beta_k)_{k\geq1}$.  If $\Ex\epsilon^4<\infty$ and $\gB=2\sum_{k=0}^\infty(k+1)\beta_k<\infty$, then
  $\sup_{\So\in\Socwr}P(\Omega^c)\leq 91 n^{-1}\sqrt{\gB}\big[(\Ex\epsilon^4)^{1/4}+\Sor\gA/\sigma\big]^2$.
\end{lem}

Considering the event $\Omega$
given in \eqref{mo:me:le:ms:omega} it is easily seen that
$\cV\subset\Omega$ and hence, taking into account the last assertion together with Proposition
\ref{dd:r:pr:dd:as1}, the conditions \ref{mo:me:as:est:A1},
\ref{mo:me:as:est:A2} and  \ref{mo:me:as:est:A3} are satisfied.  Thereby, the next assertion is an immediate consequence of  Proposition \ref{an:pr:dd:hpen} and we omit its proof.

\begin{theo}\label{id:r:t:fad}Under the assumptions of Proposition \ref{dd:r:co:dd:as1} and the condition \eqref{dd:r:pr:dd:as1:e1}. Select the dimension parameter
  $\hDi$ as given by \eqref{mo:me:de:whm} with  $\hpen:= 1152 \maxnormsup^2\Di n^{-1}\widehat\sigma_Y^2$. There exists a numerical constant $C$ and a finite constant $\zeta(\Sor\gA,\rNoL,\maxnormsup,\Ex\epsilon^6)$ depending only on the  quantities $\Sor\gA$, $\rNoL$, $\maxnormsup$ and $\Ex\epsilon^6$ such that for all $n\geq n_{\aSy}$
  with $\oRa[n_{\aSy}]\leq 1$ we have
  \begin{eqnarray*}
    \Rif{\hDiSo[\hDi]}{\Socwr}\leq C[\Sor^2\vee
    \maxnormsup^2\sigma_Y^2\vee
    \zeta(\Sor\gA,\rNoL,\maxnormsup,\Ex\rNo^{6})]\; \oRa.
  \end{eqnarray*}
\end{theo}
We shall emphasise that the last assertion establishes the
minimax-optimality of the  fully data-driven estimator $\hDiSo[\hDi]$
over all classes $\Socwr$. Therefore, the estimator is called adaptive.
%%% Local Variables: 
%%% mode: latex
%%% TeX-master: "_0DP_NPE_dep"
%%% End: 

%======================================================================================================================
%                                                                 
% Title: Simulations
%
% ======================================================================================================================
\section{Simulation study}\label{s:sim}
In this section we illustrate the performance of the proposed data-driven estimation procedure by means of a simulation study. As competitors we consider two widely used approaches, namely model selection and cross-validation, which we briefly introduce next. 
Following a model selection approach (see for example \cite{ComteRozenholc2002} in the context of dependent data) the dimension parameter is selected as following
\begin{equation*}
\widehat\Di_{MS}:=\argmin_{1\leq m\leq \DiMa}\set{-\HnormV{\hDiSo[m]}^2+cmn^{-1}\widehat\sigma_Y^2}.
\end{equation*}
We shall emphasize that this procedure relies on the contrast $-\HnormV{\hDiSo[m]}^2$ rather than $\contr$ (see equation \eqref{mo:me:de:wtm}) used in the approach studied in this paper. Moreover, the penalty term in both selection procedures involves a constant $c$ which has been calibrated in advance by a simulation study. A popular alternative provides a cross validation approach. Exploiting that the estimated coefficients satisfy $\hfSo_j=n^{-1}\sum_{i=1}^n\psi_j(Z_i)$, for $j\geq 1$, we consider the cross validation criterium given  by 
\begin{equation*}
CV(m):=\int_{[0,1]}\hDiSo[m]^2(x)dx-\frac{2}{n(n-1)}\sum_{i=1}^n\sum_{j=1}^\Di\sum_{k\ne i}\psi_j(Z_k)\bas_j(Z_i).
\end{equation*}
The dimension parameter is then selected as $\widehat\Di_{CV}=\argmin_{1\leq m\leq \DiMa}CV(m)$. Considering the orthonormal series estimator $\hDiSo[\Di]$ we denote by $\hDiSo[MS]:=\hDiSo[\widehat\Di_{MS}]$ and $\hDiSo[CV]:=\hDiSo[\widehat\Di_{CV}]$ the fully data-driven estimator based on a dimension parameter choice using the model selection and the cross-validation approach, respectively. Moreover, $\hDiSo[\hDi]$ denotes the orthogonal series estimator with $\hDi$ given as in  \eqref{mo:me:de:whm}.  In addition we compare the three fully data-driven estimators with the oracle estimator $\hDiSo[O]:=\hDiSo[\Di_o]$ where the dimension parameter $\aDi$ minimizes the integrated squared error (ISE), that is $\Di_o:=\argmin_{m\geq 1}\HnormV{\hDiSo[\Di]-\So}$. Obviously this choice is not feasible in practice.

In the following we report the performance of the four estimation procedures given independent as well as dependent observations. Therefore we make use of the framework introduced by \cite{GannazWintenberger2010} which has also been studied, for example, by \cite{BertinKlutchnikoff2014}. In the simulations we generate observations $Z_1,\dotsc,Z_n$ according to the following three different weak-dependence cases with the same marginal absolutely continuous distribution $F$.
\begin{description}
\item[Case 1] The $Z_i$ are given by $F^{-1}(U_i)$  for $1\leq i \leq n$ on $[0,1]$ where the $\set{U_i}_{i=1}^n$ are i.i.d. uniform random variables on $[0,1]$. 
\item[Case 2] The $Z_i$ are given by $F^{-1}(G(Y_i))$ where $G(y):=\tfrac{2}{\pi}\arcsin(\sqrt y)$ and the $Y_i$ are defined by $Y_1=G^{-1}(U_1)$ and recursively, for any $i\geq2$, $Y_i=T(Y_{i-1})$ with $T(y)=4y(1-y)$.
\item[Case 3] The $Z_i$ are given by $F^{-1}(G(Y_i))$ where $G$ is the marginal distribution of $Y_i$ (see for details \cite{GannazWintenberger2010}) and the $Y_i$, $i\in\Zz$ is given by 
\begin{equation*}
Y_i=2(Y_{i-1}+Y_{i+1})/5+5\zeta_i/21,
\end{equation*}
with $\set{\zeta_i}_{i\in\Zz}$ is an i.i.d. sequence of Bernoulli variables with parameter 1/2. The computation of $Z_i$'s variable is based on the method developed in \cite{DoukhanTruquet2007}.
\end{description}

Throughout the simulation study we consider the orthogonal series estimator based on the trigonometric basis. We repeat the estimation procedure for each of the four dimension selection procedures on 501 generated samples of size $n=$100, 1000, 10000. However we present only the results for $n=1000$ since in the other cases the findings were similar. 

\subsection{Non-parametric density estimation}\label{s:sim:iid:d}
We consider the estimation of two different density functions. The first one is a mixture of  two Gaussian distributions, that is  
\begin{equation*}
f_1(x)=C\left(\frac{3}{10}\phi_{0.5;0.1}(x)+\frac{1}{4}\phi_{0.7;0.06}(x)\right)\indicset{[0,1]}
\end{equation*}
where $\phi_{\mu;\sigma}$ stands for the density of a normal distribution with mean $\mu$ and standard deviation $\sigma$. The second one is defined by
\begin{equation*}
f_2(x)=C\left(4(1+|5(x-1/2)|)\right)^{-3/2}\indicset{[0,1]}.
\end{equation*}
In the both cases the numerical constant $C$ is the normalizing factor. The observations $X_1,\dotsc,X_n$ are generated according to the three cases of weak-dependence with the same marginal density $f_1$ or $f_2$. 

Figure \ref{fig:dd:density:fct1} and \ref{fig:dd:density:fct2} represent the overall behaviour of the data-driven estimator $\hDiSo[\hDi]$ of the density functions $f_1$ and $f_2$, respectively, for the three considered cases of weak-dependence. More precisely, in each figure the point-wise median and the 5\% and 95\% point-wise percentile are depicted. The quality of the estimator is visually reasonable. In addition Table \ref{sim:dd:density} reports the empirical mean and standard deviation of the ISE over the 501 Monte-Carlo repetitions.  As expected the oracle estimator $\hDiSo[O]$ outperforms the data-driven estimators. However, the increase of  the estimation error for the data-driven procedures is rather small. Moreover  the data-driven estimator $\hDiSo[\hDi]$ studied in this paper and the model selection based estimator $\hDiSo[MS]$ perform better than the cross validation procedure for both densities and all three cases of weak-dependence. Surprisingly, the selected values $\hDi$ and $\hDi_{MS}$ coincided in at least four out of the 501 Monte-Carlo repetitions for each density and each of three cases of weak-dependence, which explains the identical values in Table \ref{sim:dd:density}.  

\begin{figure}[!ht]
\centering
\includegraphics[scale=.8]{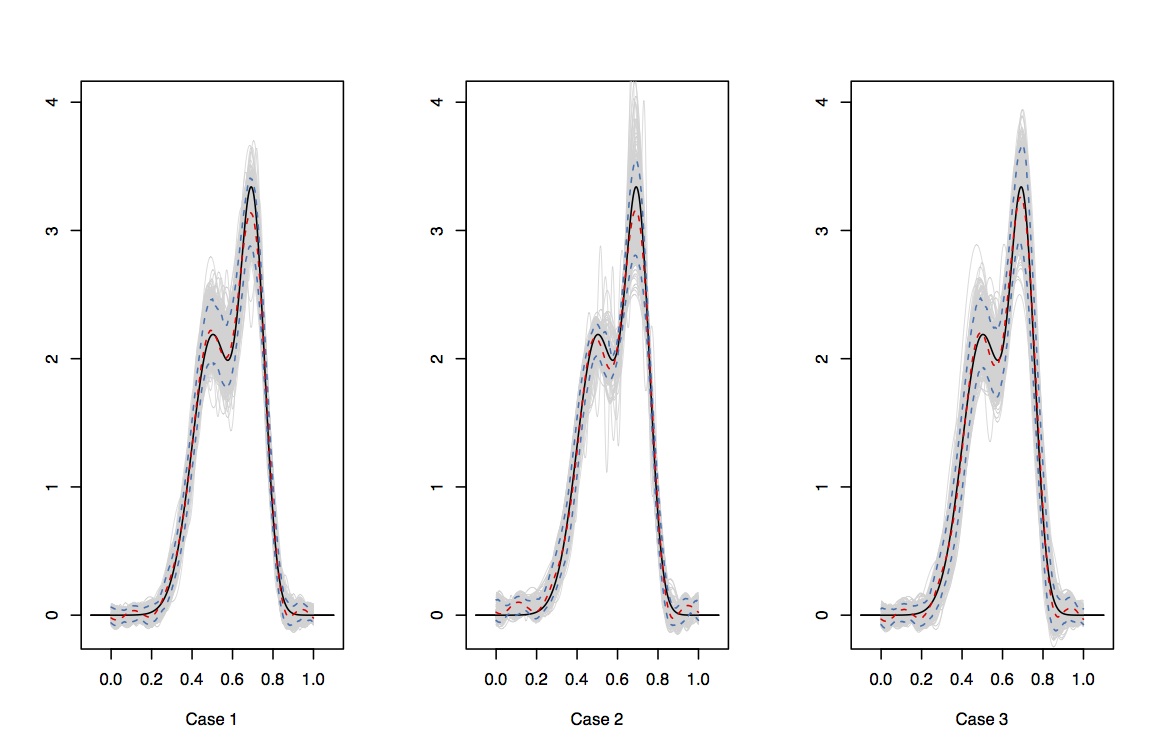}
\caption{The grey graphs depict the Monte-Carlo realisations of the data-driven estimator $\hDiSo[\hDi]$ for the density $f_1$ in the three cases of weak-dependence. The solid line corresponds to the true function, the red dashed line and the blue dashed lines represent, respectively, the point-wise median and the 5\% and 95\% point-wise percentile of the 501 replications.}
\label{fig:dd:density:fct1}
\end{figure}

\begin{table}[!ht]
\begin{tabular}{c|c|c|c|c|c}
  \hline
& &$\hDiSo[O]$ &$\hDiSo[\hDi]$& $\hDiSo[MS]$ & $\hDiSo[CV]$\\ 
  \hline
\multirow{3}{*}{$f_1$}&Case 1 & 0.0112 (0.0065) & 0.0142 (0.0089) & 0.0142 (0.0089) & 0.0178 (0.0140) \\ 
&Case 2 & 0.0102 (0.0084) & 0.0129 (0.0123) & 0.0128 (0.0119) & 0.0151 (0.0155) \\ 
&Case 3 & 0.0188 (0.0138) & 0.0213 (0.0148) & 0.0213 (0.0148) & 0.0242 (0.0169) \\ \hline
\multirow{3}{*}{$f_2$}&Case 1 & 0.0110 (0.0037) & 0.0153 (0.0053) & 0.0153 (0.0053) & 0.0159 (0.0076) \\ 
&Case 2 & 0.0123 (0.0071) & 0.0177 (0.0110) & 0.0178 (0.0108) & 0.0232 (0.0197) \\ 
&Case 3 & 0.0158 (0.0071) & 0.0210 (0.0087) & 0.0211 (0.0087) & 0.0223 (0.0118) \\ 
   \hline
\end{tabular}
\caption{Empirical mean (and standard deviation) of the ISE over the 501 Monte-Carlo simulations of sample of size $n=1000$ for the oracle and the three different data-driven estimators of the densities $f_1$ and $f_2$ in the three cases of weak-dependence.}
\label{sim:dd:density}
\end{table}

\begin{figure}[!ht]
\centering
\includegraphics[scale=.8]{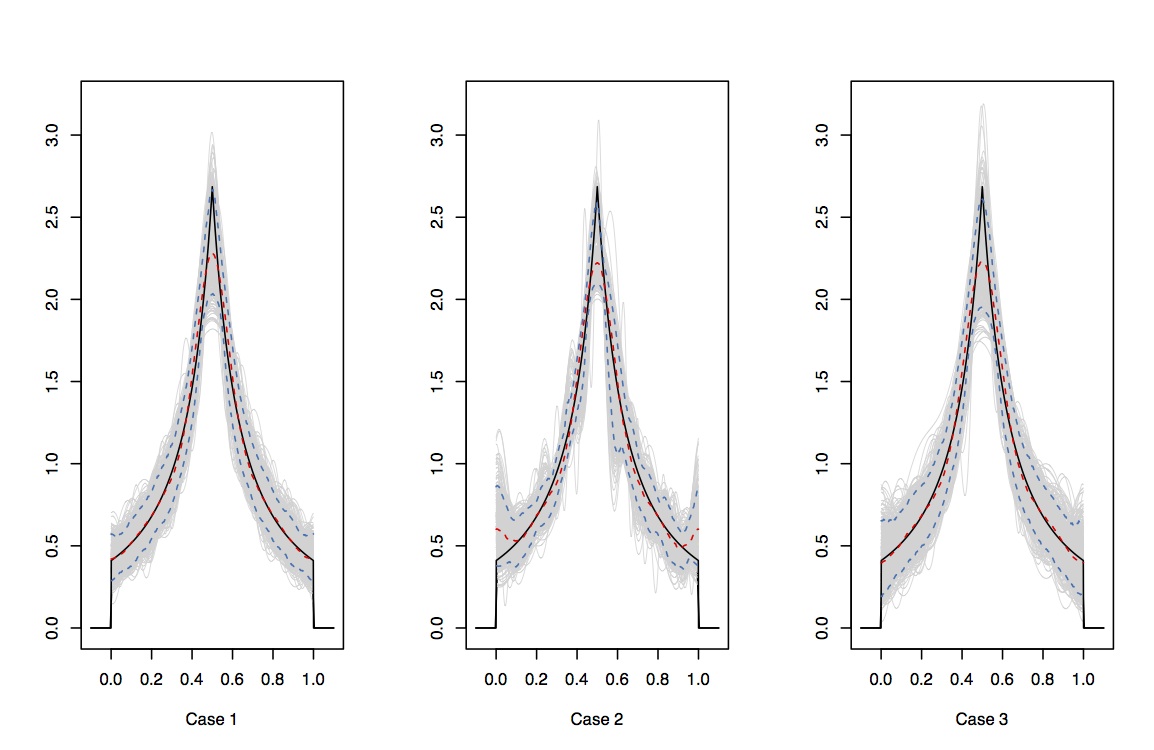}
\caption{The grey graphs depict the Monte-Carlo realisations of the data-driven estimator $\hDiSo[\hDi]$ for the density $f_2$ in the three cases of weak-dependence. The solid line corresponds to the true function, the red dashed line and the blue dashed lines represent, respectively, the point-wise median and the 5\% and 95\% point-wise percentile of the 501 replications.}
\label{fig:dd:density:fct2}
\end{figure}

\subsection{Non-parametric regression estimation}\label{s:sim:reg}
Two different regression functions are considered. The first one is a Doppler function 
\begin{equation*}
f_1(x)=(x(1-x))^{1/2}\sin\big(\tfrac{2.6\pi}{x+0.3}\big)\indicset{[0,1]}
\end{equation*}
and the second one is a mixture of a sinus function and a indicator function defined by  
\begin{equation*}
f_2(x)=\sin(4x)\indicset{[0,1/4]} +\indicset{]1/4,1]}.
\end{equation*}
In the both cases the error terms are independently and identically standard normally distributed and the noise level is set to $\sigma=0.5$. The explanatory random variables $U_1,\dotsc,U_n$ are generated according to the three cases of weak-dependence with identical marginal uniform distribution on the interval $[0,1]$. 

\begin{table}[!ht]
\begin{tabular}{c|c|c|c|c|c}
  \hline
& & $\hDiSo[O]$ &$\hDiSo[\hDi]$& $\hDiSo[MS]$ & $\hDiSo[CV]$\\  
  \hline
\multirow{3}{*}{$f_1$}&Case 1 &0.0306 (0.0091) & 0.0369 (0.0111) & 0.0369 (0.0111) & 0.0340 (0.0099) \\ 
&Case 2 & 0.0309 (0.0116) & 0.0375 (0.0146)  & 0.0375 (0.0146) & 0.0343 (0.0122) \\ 
&Case 3 & 0.0332 (0.0098) & 0.0392 (0.0109) & 0.0392 (0.0109) & 0.0370 (0.0106) \\ \hline
\multirow{3}{*}{$f_2$}&Case 1 & 0.0251 (0.0054) & 0.0318 (0.0081) & 0.0318 (0.0081) & 0.0354 (0.0122) \\ 
&Case 2 & 0.0235 (0.0064) & 0.0310 (0.0098) & 0.0310 (0.0098) & 0.0366 (0.0137) \\ 
&Case 3 & 0.0297 (0.0091) & 0.0372 (0.0139) & 0.0372 (0.0139) & 0.0388 (0.0133) \\ 
   \hline
\end{tabular}
\caption{Empirical mean (and standard deviation) of the ISE over the 501 Monte-Carlo simulations of sample of size $n=1000$ for the oracle and the three different data-driven estimators of the regressions $f_1$ and $f_2$ in the three cases of weak-dependence.}
\label{sim:dd:reg}
\end{table}

Figure \ref{fig:dd:reg:fct1} and \ref{fig:dd:reg:fct2} represent the overall behaviour of the data-driven estimator $\hDiSo[\hDi]$ of the regression functions $f_1$ and $f_2$, respectively, for the three considered cases of weak-dependence. The quality of the estimator is again visually reasonable. As in the density estimation case, the Table \ref{sim:dd:reg} reports the empirical mean and standard deviation of the ISE over the 501 Monte-Carlo repetitions. The findings are the same as for the density estimation problem with the only exception that for the regression function $f_1$ the cross validation approach performs slightly better than the other two data-driven procedures. We shall emphasize that again the selected values $\hDi$ and $\hDi_{MS}$ coincided in at least 99\% of the Monte-Carlo repetitions for each regression function and each of three cases of weak-dependence. This explains the identical value in Table \ref{sim:dd:reg} for the model selection based estimator $\hDiSo[MS]$ and the data-driven estimator $\hDiSo[\hDi]$ studied in this paper.

\begin{figure}[!ht]
\centering
\includegraphics[scale=.8]{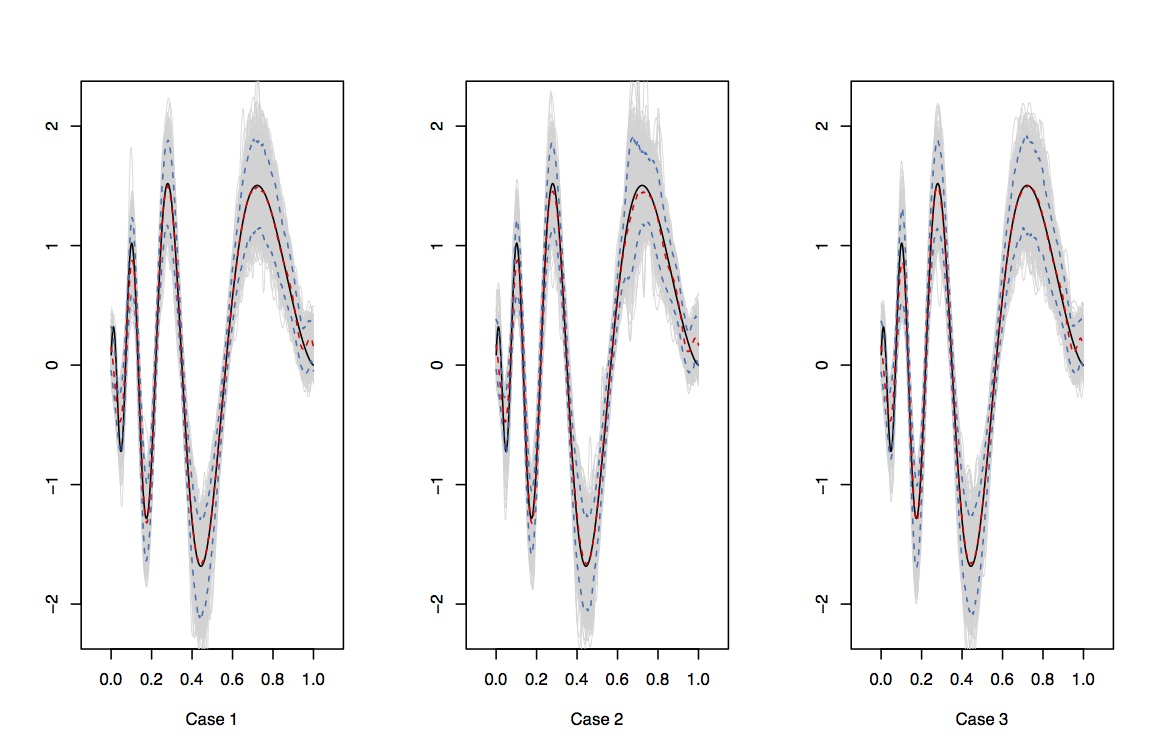}
\caption{The grey graphs depict the Monte-Carlo realisations of the data-driven estimator $\hDiSo[\hDi]$ for the regression $f_1$ in the three cases of weak-dependence. The solid line corresponds to the true function, the red dashed line and the blue dashed lines represent, respectively, the point-wise median and the 5\% and 95\% point-wise percentile of the 501 replications.}
\label{fig:dd:reg:fct1}
\end{figure}
\begin{figure}[!ht]
\centering
\includegraphics[scale=.8]{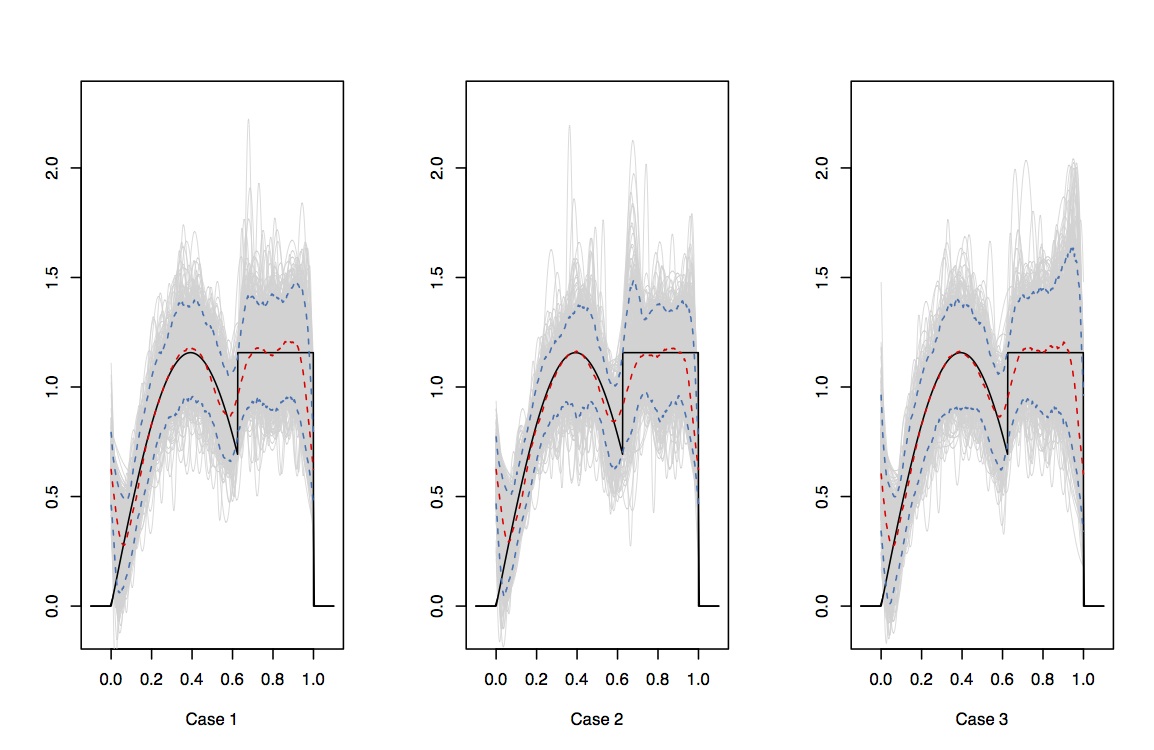}
\caption{The grey graphs depict the Monte-Carlo realisations of the data-driven estimator $\hDiSo[\hDi]$ for the regression $f_2$ in the three cases of weak-dependence. The solid line corresponds to the true function, the red dashed line and the blue dashed lines represent, respectively, the point-wise median and the 5\% and 95\% point-wise percentile of the 501 replications.}
\label{fig:dd:reg:fct2}
\end{figure}
%%% Local Variables:
%%% mode: latex
%%% TeX-master: "_0DP_NPE_dep"
%%% End:

\paragraph{Conclusions and perspectives.} In this work we present a data-driven non-parametric estimation procedure of a density and a regression function in the presence of dependent data that can attain minimax-optimal rates for independent data. Obviously, the data-driven non-parametric estimation in errors in variables models as, for example, deconvolution problems or instrumental variable regressions, are only one amongst the many interesting questions for further research and we are currently exploring this topic.
\paragraph{Acknowledgments.}
This work was supported by the IAP research network no.\ P7/06 of the
Belgian Government (Belgian Science Policy), by the "Fonds Sp\'eciaux
de Recherche'' from the Universit\'e catholique de Louvain and by the
ARC contract 11/16-039 of the "Communauté française de Belgique",
granted by the Académie universitaire Louvain. 
% --------------------------------------------------------------------
% <<Appendix>>
% --------------------------------------------------------------------
\setcounter{subsection}{0}
\appendix
\section{Appendix:  Proofs of Section \ref{s:mo}}\label{a:mo}
%======================================================================================================================
%                                                                 
% Title: Proofs of section Model assumptions
% Author: Jan JOHANNES, Nicolas Asin
% 
% Email: jan.johannes@ensai.fr nicolas.asin@uclouvain.be
% Date: %%ts latex start%%[2015-03-19 Thu 00:22]%%ts latex end%%
% Main-TeX-File: t in der form "Name"
%
% =====================================================================================================================

\begin{proof}[Proof of Lemma \ref{an:le:dd:re}]Keeping in mind the identity $\HnormV{\hDiSo[k]-\So}^2=\HnormV{\hDiSo[k]-\DiSo[k]}^2+\HnormV{\DiSo[k]-\So}^2$
 for any $ k\in\Nz$, we obtain:
\begin{multline}\label{an:le:dd:re:pr:e1}
 \Ex\left(\indicset{\Omega^c}\HnormV{\hDiSo[\tDi]-\So}^2\right)= \Ex\left(\indicset{\Omega^c}\set{\HnormV{\hDiSo[\tDi]-\DiSo[\tDi]}^2+\HnormV{\DiSo[\tDi]-\So}^2}\right)\\
 \leq \Ex\left(\indicset{\Omega^c}\set{\HnormV{\hDiSo[n]-\DiSo[n]}^2+\HnormV{\So}^2}\right)
\end{multline}
since $\HnormV{\hDiSo[k]-\DiSo[k]}^2\leq \HnormV{\hDiSo[n]-\DiSo[n]}^2$ and  $\HnormV{\DiSo[k]-\So}^2\leq\HnormV{\So}^2$ for all $1\leq k\leq \DiMa$. Considering the first right hand
side term we have 
\begin{equation}\label{an:le:dd:re:pr:e2}
\Ex\left(\HnormV{\hDiSo[n]-\DiSo[n]}^2\indicset{\Omega^c}\right)
\leq \Ex\big\{\vectp{\HnormV{\hDiSo[n]-\DiSo[n]}^2-\pen[n]/6}\big\}+\tfrac{1}{6}\pen[\DiMa]P(\Omega^c)
\end{equation}
The assertion follows now by combination of \eqref{an:le:dd:re:pr:e1}
and \eqref{an:le:dd:re:pr:e2} together with the conditions
\ref{mo:me:as:est:A1} and \ref{mo:me:as:est:A2}, and
$\HnormV{\So}^2\leq\Sor^2$, for all $\So\in\Socwr$, which completes the proof.
\end{proof}

%%% Local Variables: 
%%% mode: latex
%%% TeX-master: "_0DP_NPE_dep"
%%% End: 
 
\section{Appendix:  Proofs of Section \ref{s:id}}\label{a:id}
%======================================================================================================================
%                                                                 
% Title: Proofs of section Independent observations
% Author: Jan JOHANNES, Nicolas Asin
% 
% Email: jan.johannes@ensai.fr nicolas.asin@uclouvain.be
% Date: %%ts latex start%%[2015-03-21 Sat 19:19]%%ts latex end%%
% Main-TeX-File: t in der form "Name"
%
% =====================================================================================================================
\subsection{Appendix:  Proofs of Section \ref{s:id:d}}\label{a:id:d}
\begin{proof}[Proof of Proposition \ref{id:d:p:ub}]
In the case of independent observations it holds obviously that
\begin{equation}\label{id:d:pr:var}
\sum_{j=1}^\Di\Var\set{\frac{1}{n}\sum_{i=1}^n
  \bas_j(X_i)}=n^{-1}\sum_{j=1}^\Di\Var\{\bas_j(X)\}\leq n^{-1}\Ex
\sum_{j=1}^\Di\bas^2_j(X)\leq n^{-1}\Di
\maxnormsup^2
\end{equation}
where we have exploited the assumption \ref{mo:no:as:ba:i}. Consequently,
we have for $n,m\geq1$ that 
\begin{equation*}
\Rif{\hDiSo}{\Socwr\cap\cD}\leq n^{-1}\Di\maxnormsup^2 +
\Sow_\Di\Sor^2\leq(\maxnormsup^2+\Sor^2)\max(\Di n^{-1},\Sow_\Di)=(\maxnormsup^2+\Sor^2)\mRa.
\end{equation*}
Keeping in mind that the  dimension
parameter  $\oDi$  given in \eqref{mo:me:de:ra}, minimises the last upper risk bound, we get \eqref{id:d:p:ub:e1} which completes the proof.
\end{proof}

\begin{proof}[Proof of Proposition \ref{id:d:p:lb}]
Given $\zeta:=\eta\min(\Sor-1,
(4\gA )^{-1})$ and $\alpha_n:=\oRa /(\oDi) \leq (n \eta)^{-1}$ based on the definition of $\eta$ we consider the function $\So:= 1+
(\zeta\alpha_n)^{1/2}\sum_{1\leq j\leq \oDi}\fSo_j\bas_j$. We will show that for any
$\theta:=(\theta_j)\in\{-1,1\}^{\oDi}$, the function $\So_\theta:=
1+ \sum_{1\leq j\leq \oDi} \theta_j\fSo_j\bas_j$ belongs to
$\Socwr\cap\cD$ and is hence a possible candidate of the density. 
 We denote by $\So_\theta^n$ the
joint density of an iid.\ $n$-sample from $\So_\theta$ and by
$\Ex_\theta$ the expectation with respect to  the joint density
$\So_\theta^n$. Furthermore, for $0<j\leq\oDi$ and each $\theta$
we introduce $\theta^{(j)}$ by $\theta^{(j)}_{l}=\theta_{l}$ for $j\ne
l$ and $\theta^{(j)}_{j}=-\theta_{j}$.  The key argument of this proof
is the following reduction scheme. If $\tSo$ denotes an estimator of
$\So$ then we conclude
\begin{multline}\label{eq:443_1}
\Rif{\tSo}{\Socwr\cap\cD} \geq
  \max_{\theta\in  \{-1,1\}^{\oDi}} \Ex_\theta\HnormV{\tSo -\So_\theta}^2\geq  \frac{1}{2^{{\oDi}}}\sum_{\theta\in \{-1,1\}^{\oDi}}\Ex_\theta\HnormV{\tSo -\So_\theta}^2\\
  \geq \frac{1}{2^{{\oDi}}}\sum_{0<j\leq\oDi}\frac{1}{2}\sum_{\theta\in
    \{-1,1\}^{\oDi}}\Bigl\{\Ex_{{\theta}}|[\tSo-\So_\theta]_j|^2+\Ex_{{\theta^{(j)}}}|[\tSo-\So_{\theta^{(j)}}]_j|^2
  \Bigr\}.
\end{multline}
by using that for each $0<j\leq\oDi$ and any
function $F: \{-1,1\}^{\oDi} \to \Rz$, it holds
\[\sum_{\theta\in\{-1,1\}^{\oDi}}\So(\theta) = \sum_{\theta\in\{-1,1\}^{\oDi}}\So(\theta^{(j)}).\]
Below we show furthermore that for all $n\geq 2$ we have
\begin{equation}\label{eq:443}
  \Bigl\{\Ex_{{\theta}}|[\tSo-\So_\theta]_j|^2+\Ex_{{\theta^{(j)}}}|[\tSo-\So_{\theta^{(j)}}]_j|^2
  \Bigr\} \geq 
  \frac{\zeta}{8}\alpha_n.
\end{equation} 
From the last lower bound and the reduction scheme, by employing the definition of $\zeta$ and $\alpha_n$, we
obtain the result \eqref{id:d:p:lb:e}, that is
\begin{equation*}
\Rif{\tSo}{\Socwr\cap\cD} \geq \frac{1}{2^{{\oDi}}}\sum_{\theta\in \{-1,1\}^{\oDi}}\sum_{0<j\leq\oDi}\frac{1}{2}\frac{\zeta}{4}\alpha_n
  = \frac{\zeta}{4}  \alpha_n \oDi =\frac{\eta}{8}\, \min(\Sor-1,
(4\gA\maxnormsup^2 )^{-1})\, \oRa.
\end{equation*}

To conclude the proof, it remains to check \eqref{eq:443} and
$\So_\theta\in \Socwr\cap\cD$ for all $\theta\in\{-1,1\}^{\oDi}$. The
latter is easily verified if $\So\in \Socwr\cap\cD$.  In order to show that
$\So\in \Socwr\cap\cD$, we first notice that $\So$ integrates to
one. Moreover, $f$ is non-negative because $\normInf{\sum_{0<j\leq\oDi}
\fSo_j\bas_j} \leq1/2$, and $\wnormV{\So}^2\leq\Sor$, which can be
realised as follows. From the assumption \ref{mo:no:as:ba:ii} it follows
\begin{equation*}\normInf{\sum_{j=1}^{\oDi} \fSo_j\bas_j}^2 \leq
  \normInf{\sum_{j=1}^{\oDi}\Sow_j\bas_j^2}\;
  \bigg(\sum_{j=1}^{\oDi}\Sow_j^{-1}\fSo_j^2\bigg)\leq
  \gA^2 \bigg(\zeta\alpha_n\sum_{j=1}^{\oDi}\Sow_j^{-1}\bigg).
\end{equation*}
Since $\Sow^{-1}$ is monotonically increasing the definition of $\zeta$, $\alpha_n$
and $\eta$ implies
\begin{equation}\label{pr:theo:lower:n:e1}
 \normInf{\sum_{j=1}^{\oDi} \fSo_j\bas_j}^2 \leq  \gA^2\zeta\alpha_n {\oDi}\Sow_{\oDi}^{-1}
\leq (\eta/4) \Sow_{\oDi}^{-1} \alpha_n \oDi=
\eta\Sow_{\oDi}^{-1}\oRa/4\leq 1/4
\end{equation}
as well as $\wnormV{\So}^2 \leq 1+\zeta \Sow_{\oDi}^{-1}\alpha_n \oDi \leq { 1+\zeta/\eta\leq r}$. It remains to show~\eqref{eq:443}. Consider the Hellinger affinity
$\rho(\So_{\theta}^n,\So_{\theta^{(j)}}^n)= \int
\sqrt{\So_{\theta}^n}\sqrt{\So_{\theta^{(j)}}^n}$, then we obtain
for any estimator $\tSo$ of $\So$ that
\begin{align*}
  \rho(\So_{\theta}^n,\So_{\theta^{(j)}}^n)
  &\leqslant \big( \int
  \frac{|[\tSo-\So_{\theta^{(j)}}]_j|^2}{|[\So_\theta-\So_{\theta^{(j)}}]_j|^2}
  \So_{\theta^{(j)}}^n\big)^{1/2} + \big( \int
  \frac{|[\tSo-\So_{\theta}]_j|^2}{|[\So_\theta-\So_{\theta^{(j)}}]_j|^2}
  \So_{\theta}^n\big)^{1/2}.
\end{align*}
Rewriting the last estimate we obtain
\begin{equation}\label{pr:theo:lower:n:e2}
  \Bigl\{\Ex_{{\theta}}|[\tSo-\So_\theta]_j|^2+\Ex_{{\theta^{(j)}}}|[\tSo-\So_{\theta^{(j)}}]_j|^2 \Bigr\} \geq  \frac{1}{2} |[\So_\theta-\So_{\theta^{(j)}}]_j|^2 \rho^2(\So_{\theta}^n,\So_{\theta^{(j)}}^n). \end{equation}
Next we bound from below the Hellinger affinity $\rho(\So_{\theta}^n,\So_{\theta^{(j)}}^n)$. Therefore, we consider first the Hellinger distance 
\begin{multline*}H^2(\So_{\theta},\So_{\theta^{(j)}}) 
= \int\frac{|\So_{\theta}-\So_{\theta^{(j)}}|^2}{\big(\sqrt\So_{\theta}+\sqrt\So_{\theta^{(j)}}\big)^2}\leq
\frac{1}{2} \HnormV{\So_{\theta}-\So_{\theta^{(j)}} }^2 
= 2
  |\fSo_j|^2 \leq { \frac{2\zeta}{\eta\, n}},
\end{multline*}
where we have used that $\alpha_n\leq(n\eta)^{-1}$ and $\So_\theta\geq 1/2$ because
$|\sum_{0<j\leq\oDi} [\So_\theta]_j\bas_j|\leq 1/2$ (see \eqref{pr:theo:lower:n:e1}). {Therefore, the definition of $\zeta$ implies
$H^2(\So_{\theta},\So_{\theta^{(j)}}) \leq 2/n$.  By using the
independence, i.e.,
$\rho(\So_{\theta}^n,\So_{\theta^{(j)}}^n)=\rho(\So_{\theta},\So_{\theta^{(j)}})^n$,
together with the identity
$\rho(\So_{\theta},\So_{\theta^{(j)}})=1-\frac{1}{2}H^2(\So_{\theta},\So_{\theta^{(j)}})$
it follows $\rho(\So_{\theta}^n,\So_{\theta^{(j)}}^n) \geq (1-
n^{-1})^n\geq 1/4$ for all $n\geq2$.} By combination of the last
estimate with \eqref{pr:theo:lower:n:e2} we obtain \eqref{eq:443}
which completes the proof.\end{proof}

\begin{proof}[Proof of Proposition \ref{id:d:p:co}]
Keeping in mind Remark \ref{id:ka:rem} we intend to apply Talagrand's inequality
(Lemma \ref{id:ka:l:talagrand}) where we need to compute the quantities $h$, $H$ and
$v$ verifying the three required inequalities. Consider first $h$
where due to  the assumption \ref{mo:no:as:ba:i}
\begin{equation}\label{prop:adaptivity:density:e1}\sup_{t\in\Bz_m}\normInf{\nu_t}^2
=\normInf{\sum_{j=1}^m\bas_j^2} \leq \maxnormsup^2 m=:h^2.
\end{equation}
Consider next $H$ where
\begin{equation}\label{prop:adaptivity:density:e2}
\Ex\sup_{t\in\Bz_m}|\overline{\nu_t}|
=\big(\Ex\HnormV{\hSo_m-\So_m}^2 \big)^{1/2}=\big(\sum_{j=1}^m\Var(\hfou{\So}_j) \big)^{1/2}
\leq \big[\Di n^{-1}\maxnormsup^2\big]^{1/2}=:H.
\end{equation}
Consider finally $v$. Due to assumption \ref{mo:no:as:ba:ii}  for all $\So\in\Socwr$, we have 
\begin{equation}\label{prop:adaptivity:density:e3}
        \sup_{t\in\Bz_m}\Ex|\nu_t(\dOb)|^2
        =\sup_{t\in\Bz_m}\Ex|\sum_{j=1}^m\fou{t}_j\bas_j(\dOb)|^2
\leq
	\normInf{\So}\leq \Sor\gA=:v.
\end{equation}
The assertion follows from Lemma \ref{id:ka:l:talagrand} by using the quantities $h$, $H$ and
$v$ given in \eqref{prop:adaptivity:density:e1},
\eqref{prop:adaptivity:density:e2} and
\eqref{prop:adaptivity:density:e3}, respectively and by employing the definition of $\zeta$,
which completes the proof.
\end{proof}

\subsection{Appendix:  Proofs of Section \ref{s:id:r}}\label{a:id:r}
\begin{proof}[Proof of Proposition \ref{id:r:p:ub}]In the case of independent observations it holds obviously that
\begin{equation}\label{id:r:pr:var}
\sum_{j=1}^\Di\Var\set{\frac{1}{n}\sum_{i=1}^n\rOb_i \bas_j(\rRe_i)}
\leq n^{-1}\Ex
\rOb^2\sum_{j=1}^\Di\bas^2_j(\rRe)\leq n^{-1}\Di
\maxnormsup^2 (\sigma^2+\HnormV{\So}^2)
\end{equation}
where we have exploited assumption \ref{mo:no:as:ba:i} and
$\sigma_Y^2:=\Ex Y^2=\sigma^2+\HnormV{\So}^2$. Keeping mind that $\HnormV{\So}^2\leq\Sor^2$
for all $\So\in\Socwr$ 
we have for $n,m\geq1$, $\Rif{\hDiSo}{\Socwr}\leq 
(\maxnormsup^2 (\sigma^2+\Sor^2)+\Sor^2)\mRa.$

Employing further that the  dimension
parameter  $\oDi$  given in \eqref{mo:me:de:ra}  minimises the last upper risk bound, ie., the term
$\mRa=\max(\Di n^{-1},\Sow_\Di)$, with
respect to the dimension parameter, we obtain \eqref{id:r:p:ub:e1} which completes the proof.
\end{proof}

\begin{proof}[Proof of Proposition \ref{id:r:p:lb}]
Given $\zeta:=\eta\min(\Sor^2,\sigma^2/2)$ and $\alpha_n:=\oRa /\oDi \leq (n \eta)^{-1}$ due \eqref{id:r:p:lb:e1} we consider the function $\So:=(\zeta\alpha_n)^{1/2}\sum_{j=1}^{\oDi}\bas_j$. We will show that for any
$\theta:=(\theta_j)_{j=1}^{\oDi}\in\{-1,1\}^{\oDi}$, the function $\So_\theta:=
\sum_{j=1}^{\oDi} \theta_j[\So]_j\bas_j$ belongs to
$\Socwr$ and is hence a possible candidate of the regression
function. For a fixed $\theta$ and under the hypothesis that the
regression function is $\So_\theta$, we denote by $P_\theta^n$ the
 joint distribution of the observation $\{(Y_i,U_i)\}_{i=1}^n$ and by
$\Ex_\theta$ the expectation with respect to  this
distribution. Furthermore, for $1\leq j\leq\oDi$ and each $\theta$
we introduce $\theta^{(j)}$ by $\theta^{(j)}_{l}=\theta_{l}$ for $j\ne
l$ and $\theta^{(j)}_{j}=-\theta_{j}$.  The key argument of this proof
is the following reduction scheme \eqref{eq:443_1} . From the lower bound \eqref{eq:443} and the reduction scheme \eqref{eq:443_1}, by employing the definition of $\zeta$ and $\alpha_n$, we
obtain the result \eqref{id:r:p:lb:e}, that is
\begin{equation*}
\Rif{\tSo}{\Socwr} \geq \frac{1}{2^{{\oDi}}}\sum_{\theta\in \{-1,1\}^{\oDi}}\sum_{j=1}^{\oDi}\frac{1}{2}\frac{\zeta}{2}\alpha_n
  = \frac{\zeta}{4}  \alpha_n \oDi =\frac{\eta}{8}\, \min(2\Sor^2,
\sigma^2)\, \oRa.
\end{equation*}

To conclude the proof, it remains to check \eqref{eq:443} and
$\So_\theta\in \Socwr$ for all $\theta\in\{-1,1\}^{\oDi}$. The
latter is easily verified if $\So\in \Socwr$, which can be realised
as  follows. By applying successively that $\Sow$ is monotonically
increasing, that $\oRa \Sow_{\oDi}\leq \eta^{-1}$ due
\eqref{id:r:p:lb:e1} and, hence $\zeta\alpha_n
\oDi\Sow_{\oDi}=\zeta \oRa \Sow_{\oDi} \leq \Sor^2$ 
we obtain $\wnormV{\So}^2 \leq \zeta\alpha_n \oDi\Sow_{\oDi}
\leq \Sor^2$ which proves the claim. 

Next we bound from below the Hellinger affinity
$\rho(P_{\theta}^n,P_{\theta^{(j)}}^n)$ using the well-known
relationship $\rho(P_{\theta}^n,P_{\theta^{(j)}}^n) \geq 1- (1/2) KL(P_{\theta}^n,P_{\theta^{(j)}}^n)$
between the Kullback-Leibler divergence and the Hellinger affinity. We
will show that $KL(P_{\theta}^n,P_{\theta^{(j)}}^n)\leq 1$, and hence
$\rho(P_{\theta}^n,P_{\theta^{(j)}}^n) \geq 1/2$ which together with
\eqref{pr:theo:lower:n:e2} and
$|[\So_\theta-\So_{\theta^{(j)}}]_j|^2= 4[\So]_j^2=4\zeta\alpha_n$
implies \eqref{eq:443}. Therefore, consider the Kullback-Leibler
divergence between $P_{\theta}^n$ and  $ P_{\theta^{(j)}}^n$. Recall, that  for a fixed $\theta$ and under the hypothesis that the
regression function is $\So_\theta$, the observations
$\{Y_i\}_{i=1}^n$ are conditional independent  given the regressors $\{U_i\}_{j=1}^n$
and for each $1\leq i\leq n$ the conditional distribution of
$Y_i$  given the regressor $U_i$ is normal with conditional mean $f_\theta(U_i)$
and conditional variance $\sigma^2$. Therefore, we have \[\log
\frac{dP_{\theta}^n(\{(Y_i,U_i)\}_{i=1}^n)}{dP_{\theta^{(j)}}^n(\{(Y_i,U_i)\}_{i=1}^n)}=
\sum_{i=1}^n \frac{2\zeta\alpha_n}{\sigma^2} \bas_j^2(U_i) + \sum_{i=1}^n \frac{2\theta_j(\zeta\alpha_n)^{1/2}}{\sigma^2}\bas_j(U_i)(Y_i-\So_\theta(U_i)).\]
Taking the expectation $\Ex_{\theta}$ with respect to $P_{\theta}^n$ leads to
$KL(P_{\theta}^n,P_{\theta^{(j)}}^n)=  2\zeta\alpha_nn/\sigma^2$. By
employing that $\alpha_n n\leq1/\eta$ and $\zeta/(\eta\sigma^2)\leq
1/2$ we obtain that  $KL(P_{\theta}^n,P_{\theta^{(j)}}^n)\leq1$ which
shows the claim and  completes the proof.
\end{proof}

\begin{proof}[Proof of Proposition \ref{id:r:p:co}]
The key argument of the next
assertion is again Talagrand's inequality.  However, a direct
application employing  $\sup_{t\in\Bz_m} |\overline{\nu_t}|^2 = 
\HnormV{\hSo_m-\So_m}^2$  with $\overline{\nu_t}=\frac{1}{n}\sum_{i=1}^n\left[\nu_t(\rNo_i,\rRe_i)-\Ex\left(\nu_t(\rNo_i,\rRe_i)\right)
\right]$ and $\nu_t(\rNo,\rRe)=\sum_{j=1}^{\Di}\fou{t}_j(\rNoL\rNo+\So(\rRe))\bas_j(\rRe)$ is
not possibly noting that $\rNo$ and hence $\nu_t$ are generally not uniformly
bounded. Therefore, let us introduce  $\rNo^b:=\rNo\indic{|\rNo|\leq
  n^{1/4}}-\Ex\rNo\indic{|\rNo|\leq  n^{1/4}}$ and $\rNo^u:=\rNo-\rNo^b=\rNo\indic{|\rNo|>
   n^{1/4}}-\Ex\rNo\indic{|\rNo|> n^{1/4}}$. Setting 
 $\overline{\nu_t^b}=\frac{1}{n}\sum_{i=1}^n\big[\nu_t(\rNo_i^b,\rRe_i)-\Ex\big(\nu_t(\rNo_i^b,\rRe_i)\big)\big]$, $\nu^u_t(\rNo^u,\rRe):=\sum_{j=1}^{\Di}\fou{t}_j\rNoL\rNo^u\bas_j(\rRe)$
and  $\overline{\nu_t^u}=\frac{1}{n}\sum_{i=1}^n\left[\nu_t^u(\rNo_i^u,\rRe_i)-\Ex\left(\nu_t^u(\rNo_i^u,\rRe_i)\right)
\right]$ we have obviously
$\overline{\nu_t}=\overline{\nu_t^b}+\overline{\nu_t^u}$. 
Consequently, exploiting the elementary inequality
$|{\overline{\nu_t}}|^2\leq
2\set{|\overline{\nu_t^b}|^2+|\overline{\nu_t^u}|^2}$ follows that
\begin{multline}\label{id:r:dd:e1}
\Ex\vectp{\max_{\oDi\leq \Di\leq \DiMa}\{\HnormV{\hDiSo-\DiSo}^2-
  \tfrac{\pen}{6}\}}
\\\leq
{2}\Ex\vectp{\max_{\oDi\leq \Di\leq n}\{\sup_{t\in\Bz_\Di}|{\overline{\nu_t^b}}|^2-
  \tfrac{\pen}{12}\}}+2\Ex\sup_{t\in\Bz_n}|{\overline{\nu_t^u}}|^2.
\end{multline} 
We  bound separately each term on the rhs. of the last display.
Consider first the second right hand side term. Since
$\Ex(\rNo^{6})<\infty$ which implies that $\Ex(\rNo^2)\indic{\rNo^2>\eta}\leq\eta^{-2}\Ex(\rNo^{6})$ for all
$\eta>0$, it follows from the independence assumption and \ref{mo:no:as:ba:i}  that 
\begin{equation}\label{id:r:p:co:pr:u}
\Ex\sup_{t\in\Bz_n}|\overline{\nu_t^u}|^2
\leq 
\rNoL^2\maxnormsup^2 \Var(\rNo^u)\leq\rNoL^2\maxnormsup^2
\Ex\big(\rNo^2\indic{|\rNo|> n^{1/4}}\big)\leq n^{-1} \rNoL^2\maxnormsup^2\Ex(\rNo^{6}).
\end{equation}
In order to bound the second right hand side term in
\eqref{id:r:dd:e1}, we aim to  apply Talagrand's inequality (Lemma
\ref{id:ka:l:talagrand}) which necessitates the computation of the
quantities $h$, $H$ and $v$ verifying the required
inequalities. Consider first $h$. Let $\psi_j(e^b,u)=(\rNoL
        e^b+\So(u))\bas_j(u)$ and note that $|\rNo^b|\leq 2n^{1/4}$
by construction. Hence, employing \ref{mo:no:as:ba:i} we have
\begin{equation}\label{id:r:p:co:pr:h}
	\sup_{t\in\Bz_m}\inormV{v_t}^2= \sum_{j=1}^\Di\inormV{\psi_j^2}
        \leq  \maxnormsup^2 m (2\rNoL n^{1/4} +\inormV{\So})^2=:h^2.
\end{equation}
Next we compute the quantity  $H$, where due to assumption \ref{mo:no:as:ba:i}

\begin{equation*}
\Ex\sup_{t\in\Bz_m}|\overline{v_t^b}|^2
\leq \frac{1}{n} \Ex\big\{
(\rNoL\rNo^b_1+\So(\rRe_1))^2\sum_{j=1}^m\bas_j^2(\rRe_1)\big\}
\leq  \frac{m\maxnormsup^2}{n} \Ex
(\rNoL\rNo^b_1+\So(\rRe_1))^2.
 \end{equation*}
Exploiting $\Var\rNo^b\leq \Ex\big(\rNo^2\indic{|\rNo|>
  n^{1/4}}\big)\leq \Ex \rNo^2 =1$ and the independence between $\rNo$ and $\rRe$ we have $\Ex
(\rNoL\rNo^b_1+\So(\rRe_1))^2=\rNoL^2\Var\rNo^b_1+\HnormV{\So}^2\leq
\sigma^2+\HnormV{\So}^2 =\Ex \rOb^2=\sigma_Y^2$. Combining the bounds it
follows that
\begin{equation}\label{id:r:p:co:pr:H}
\Ex\sup_{t\in\Bz_m}|\overline{v_t^b}|\leq \big(\Ex\sup_{t\in\Bz_m}|\overline{v_t^b}|^2\big)^{1/2}\leq
n^{-1/2}m^{1/2}\maxnormsup \sigma_Y=:H.
 \end{equation}

It remains to calculate the third quantity $\nu$, where due to the
independence between $\epsilon$ and $U$ 
\begin{multline}\label{id:r:p:co:pr:v}
\sup_{t\in\Bz_m}\frac{1}{n}\sum_{i=1}^n \Var(v_t(\rNo^b_i,\rRe_i))
\leq \sup_{t\in\Bz_m}\Ex(v_t(\rNo^b_1,\rRe_1))^2\\
=\sup_{t\in\Bz_m}\{\rNoL^2\Var(\rNo^b)\Ex\big(\sum_{j=1}^m\fou{t}_j\bas_j(\rRe_1)\big)^2+\Ex\big(\So(\rRe_1)\sum_{j=1}^m\fou{t}_j\bas_j(\rRe_1)\big)^2\}\\
\leq \sup_{t\in\Bz_m}\{\rNoL^2\HnormV{t}^2+\inormV{\So}^2\HnormV{t}^2\}=\rNoL^2+\inormV{\So}^2=:\nu.
\end{multline}
Replacing in Lemma \ref{id:ka:l:talagrand} the constants $h$, $H$ and
$v$ by \eqref{id:r:p:co:pr:h}, \eqref{id:r:p:co:pr:H}  and
\eqref{id:r:p:co:pr:v} respectively, there exists a finite numerical constant $C>0$ such that
\begin{multline*}
\Ex\vectp{\sup_{t\in\Bz_m}|\overline{v_t^b}|^2- 6 \maxnormsup^2\sigma_Y^2mn^{-1} }\leq
C\bigg[\frac{\sigma^2+\inormV{\So}^2}{n}\exp\left(
    - \frac{m \maxnormsup^2\sigma_Y^2}{6(\sigma^2+\normInf{\So}^2)}\right)
 \\ +
  \frac{2 \maxnormsup^2m(\rNoL+\inormV{\So})^2}{n^{3/2}}\exp(-\frac{K}{2}n^{1/4}\frac{\sigma_Y}{\rNoL+\inormV{\So}})\bigg].
\end{multline*}
The last upper bound and
$\frac{\sigma^2+\normInf{\So}^2}{\sigma_Y^2}=\frac{\sigma^2+\normInf{\So}^2}{\sigma^2+\HnormV{\So}^2}\leq
2\left(\frac{\sigma+\normInf{\So}}{\sigma+\HnormV{\So}}\right)^2\leq2(1+\normInf{\So}/{\sigma})^2$
imply together the existence of a finite numerical constant $C>0$ such that
\begin{multline*}
\Ex\vectp{\max_{1\leq \Di\leq n}\{\sup_{t\in\Bz_m}|\overline{v_t^b}|^2 -6\maxnormsup^2\sigma_Y^2mn^{-1}\}}
\leq C\frac{\sigma^2+\inormV{\So}^2}{n}\big[\sum_{m=1}^n\exp\left(
  -\frac{m\maxnormsup^2}{12(1+\inormV{\So}/\sigma)^2}\right)\\+
n^{3/2}\maxnormsup^2\exp(-n^{1/4}\frac{K}{2(1+\inormV{\So}/\sigma)})\big]\end{multline*}
and hence, from $\inormV{\So}\leq \Sor\gA$ for all $\So\in\Socwr$ due
to assumption \ref{mo:no:as:ba:ii} there exists a finite constant $C(\Sor\gA,\rNoL,\maxnormsup)$ depending only on the
  quantities $\Sor\gA$, $\rNoL$ and $\maxnormsup$ such that 
\begin{equation*}
\sup_{\So\in\Socwr} \Ex\vectp{\max_{1\leq \Di\leq n}\{\sup_{t\in\Bz_m}|\overline{v_t^b}|^2
  -6\maxnormsup^2\sigma_Y^2mn^{-1}\}}\leq
n^{-1}C(\Sor\gA,\rNoL,\maxnormsup),\quad\mbox{for all }n\geq1.\end{equation*}
The assertion of Proposition \ref{id:r:p:co} follows now by 
combination of the last bound, \eqref{id:r:p:co:pr:u} and the decomposition \eqref{id:r:dd:e1}, which completes the proof.
\end{proof}

\begin{proof}[Proof of Lemma \ref{id:r:l:re}] We start the proof with
  the observation that
  $\cV^{c}\subset\set{\left|\frac{\widehat{\sigma}^2_Y}{\sigma^2_Y}-1\right|\geq\frac{1}{2}}$
  and, hence
\begin{equation*}
\proba{\cV^c}\leq\proba{\left|\frac{\widehat{\sigma}^2_Y}{\sigma^2_Y}-1\right|\geq\frac{1}{2}}=\proba{\left|n^{-1}\sum_{i=1}^n\left(\frac{Y_i^2}{\sigma_Y^2}-1\right)\right|\geq\frac{1}{2}}.
\end{equation*}
Since $\Ex Y_i^2=\sigma_Y^2$ and employing Tchebysheff's inequality 
\begin{equation*}
\proba{\left|n^{-1}\sum_{i=1}^n\left(\frac{Y_i^2}{\sigma_Y^2}-1\right)\right|\geq\frac{1}{2}}
\leq
\frac{4}{n\sigma_Y^4}\Ex{Y_1^4}\leq \frac{128}{n} \left((\Ex\epsilon^4)^{1/4}+\inormV{f}/\sigma\right)^4.
\end{equation*}
The assertion follows now by taking into account that $\inormV{f}\leq
\Sor\gA$ for all $\So\in\Socwr$, which completes the proof.  
\end{proof}

%%% Local Variables: 
%%% mode: latex
%%% TeX-master: "_0DP_NPE_dep"
%%% End: 
 
\section{Appendix: Proofs of Section \ref{s:dd}}\label{a:dd}
%======================================================================================================================
%                                                                 
% Title: Proofs of section Dependent observations
% Author: Jan JOHANNES, Nicolas Asin
% 
% Email: jan.johannes@ensai.fr nicolas.asin@uclouvain.be
% Date: %%ts latex start%%[2015-03-22 Sun 02:28]%%ts latex end%%
% Main-TeX-File: t in der form "Name"
%
% ======================================================================================================================
\subsection{Appendix:  Proofs of Section \ref{s:dd:d}}\label{a:dd:d}
\begin{proof}[\dr Proof of Lemma \ref{dd:pr:ub}]
Combining the assumption
\ref{mo:no:as:ba:i} and Lemma \ref{dd:le:var} we get a first bound for its variance,
  \begin{multline*}\label{dd:le:var:e6}
\sum_{j=1}^\Di\Var(\tfrac{1}{n}\sum_{i=1}^n\bas_j(\dOb_i))\leq \tfrac{1}{n} \Ex (\sum_{j=1}^\Di|\bas_j(\dOb_0)|^2\{1+4\sum_{k=1}^{n-1}b(\dOb_0)\})
\leq \maxnormsup^2\{1+4\sum_{k=1}^{n-1}\beta(\dOb_0,\dOb_k)\}  \Di n^{-1}.
\end{multline*}Then, the assertion \ref{dd:pr:ub:e1} is an immediate consequence.
\end{proof}

\begin{proof}[\dr Proof of Lemma \ref{dd:le:est:var}]
We start the proof with the observation that for any orthonormal
system $\{\bas_j\}_{j=1}^m$ we have $\HnormV{\sum_{j=1}^\Di\bas_j\otimes\bas_j}^2=\sum_{j=1}^\Di\sum_{l=1}^\Di|\HskalarV{\bas_j,\bas_l}|^2= m.$
Thereby, exploiting the assumption \ref{dd:as:de:i} it follows that 
\begin{equation}\label{dd:le:est:var:pr:e1}
\bigg|\sum_{j=1}^\Di\Cov(\bas_j(\dOb_0),\bas_j(\dOb_k))\bigg|\\
\leq \HnormV{\sum_{j=1}^\Di\bas_j\otimes\bas_j}\HnormV{f_{\dOb_0,\dOb_k}-f_{\dOb_0}\otimes f_{\dOb_k}}\leq \sqrt{m}\gamma_{f}.
\end{equation}
On the other hand side, following the proof of  Lemma \ref{dd:le:var}
there exists a function $b_k:\Rz\to[0,1]$ with  $\Ex b_k(\dOb_0)=\beta(\dOb_0,\dOb_k)$ such that
 \begin{equation}\label{dd:le:est:var:pr:e2}
\bigg|\sum_{j=1}^\Di\Cov(\bas_j(\dOb_0),\bas_j(\dOb_k))\bigg|\leq 2
\Ex(b_k(\dOb_0)\{\sum_{j=1}^\Di \bas_j^2(\dOb_0)\})\leq 2 m \maxnormsup^2 \beta(\dOb_0,\dOb_k)
 \end{equation}
where the last inequality follows from the assumption \ref{mo:no:as:ba:i}.
By combination of \eqref{dd:le:est:var:pr:e1}  and \eqref{dd:le:est:var:pr:e2}  we
obtain for any $0\leq K \leq n-1$
\begin{multline*}
\sum_{k=1}^{n-1}(n+1-k)\sum_{j=1}^\Di\Cov(\bas_j(\dOb_0),\bas_j(\dOb_k))
\leq \sqrt{m}\gamma_f  n K + 2 m \maxnormsup^2 n
\sum_{k=K+1}^{n-1}\beta(\dOb_0,\dOb_k) \\
= m n \{ \gamma_f K/\sqrt{m} + 2
\maxnormsup^2 \sum_{k=K+1}^{n-1}\beta(\dOb_0,\dOb_k)\}.
\end{multline*}
From the last bound and the assumption \ref{mo:no:as:ba:i}  we conclude that 
\begin{multline*}
   \sum_{j=1}^\Di  \Var(\sum_{i=1}^n\bas_j(\dOb_i))=\sum_{j=1}^\Di
   \sum_{i=1}^n\Var(\bas_j(\dOb_i))+ 2\sum_{j=1}^\Di \sum_{i=2}^n(n+1-i)\Cov(\bas_j(\dOb_1),\bas_j(\dOb_i)) \\
\leq n \Ex \{\sum_{j=1}^\Di\bas^2_j(\dOb_0)\} + 2
\sum_{k=1}^{n-1}(n-k)\big|
\sum_{j=1}^\Di\Cov(\bas_j(\dOb_0),\bas_j(\dOb_k))\big|\\
\leq n m \maxnormsup^2 + 2m n \{ \gamma_f K/\sqrt{m} + 2
\maxnormsup^2 \sum_{k=K+1}^{n-1}\beta(\dOb_0,\dOb_k)\}
\end{multline*}
which shows the assertion and completes the proof.\end{proof}

\begin{proof}[\dr Proof of Proposition \ref{dd:co:dd:as1}]
Following the construction presented in 
Section \ref{s:dd}
let $(\dOb_i)_{i\geq1}=(E_l,O_l)_{l\geq 1}$ and
$(\cou{\dOb}_i)_{i\geq1}=(\couE_l,\couO_l)_{l\geq1}$ be random vectors  satisfying the coupling properties
\ref{dd:as:cou1}, \ref{dd:as:cou2} and \ref{dd:as:cou3}. Let $n$, $p$ and $q$ be integers such that
$n=2pq$. Let us
introduce exactly in the same way
$(x_1,\dotsc,x_n)=(e_1,o_1,\dotsc,e_p,o_p)$ with
$e_l=(x_i)_{i\in\cI^e_{l}}$ and $o_l=(x_i)_{i\in\cI^o_{l}}$,
$l=1,\dotsc,p$. If we set further for any $x=(x_1,\dotsc,x_q)\in[0,1]^q$,
$\vec{v}_t(x):=(1/q)\sum_{i=1}^qv_t(x_i)$, 
then $  \tfrac{1}{n}\sum_{i=1}^n\nu_t(x_i)=\tfrac{1}{2}\set{\tfrac{1}{p}\sum_{l=1}^p \vec{v_t}(e_l)
      +\tfrac{1}{p}\sum_{l=1}^p  \vec{v_t}(o_l)}$.
 Thereby, it follows for $\overline{\nu_t}=(1/n)\sum_{i=1}^n\left[\nu_t(\dOb_i)-\Ex\left(\nu_t(\dOb_i)\right)
\right]=\skalarV{t,\hDiSo-\DiSo}$ that $ \overline{\nu_t}=:\tfrac{1}{2}\set{\overline{\nu_t^e}+\overline{\nu_t^o}}$.
Considering rather than $(X_i)_{i=1}^n$ the random variables $(\couX_i)_{i=1}^n$ we
introduce additionally
\begin{equation*}
  \cou{\overline{\nu_t}}=\tfrac{1}{2}\set{\tfrac{1}{p}\sum_{l=1}^p
    \{\vec{v_t}(\couE_l)-\Ex \vec{v_t}(\couE_l)\}+\tfrac{1}{p}\sum_{l=1}^p \{\vec{v_t}(\couO_l)-\Ex
    \vec{v_t}(\couO_l)\}} =:  \tfrac{1}{2}\set{\cou{\overline{\nu_t^e}}+\cou{\overline{\nu_t^o}}}.
\end{equation*}
Using successively Jensen's inequality, i.e., $|{\overline{\nu_t}}|^2\leq \tfrac{1}{2}\set{|\overline{\nu_t^e}|^2+|\overline{\nu_t^o}|^2}$ , $|a|^2\leq2\{|b|^2+|a-b|^2\}$, $\Bz_\Di\leq \Bz_{\DiMa}$ for all $1\leq
    \Di\leq  \DiMa$ it
follows that 
\begin{multline*}
\Ex\vectp{\max_{\oDi\leq \Di\leq \DiMa}\set{\normV{\hDiSo-\DiSo}^2-
  \tfrac{\pen}{6}}}\\\leq
\Ex\vectp{\max_{\oDi\leq \Di\leq\DiMa}\set{\sup_{t\in\Bz_\Di}|\cou{\overline{\nu_t^e}}|^2-
  \tfrac{\pen}{12}}}+\Ex\vectp{\sup_{t\in\Bz_{\DiMa}}|\cou{\overline{\nu_t^e}}-{\overline{\nu_t^e}}|^2}\\+
\Ex\vectp{\max_{\oDi\leq \Di\leq \DiMa}\set{\sup_{t\in\Bz_\Di}|\cou{\overline{\nu_t^o}}|^2-
  \tfrac{\pen}{12}}}+\Ex\vectp{\sup_{t\in\Bz_{\DiMa}}|\cou{\overline{\nu_t^o}}-{\overline{\nu_t^o}}|^2}.
\end{multline*}  
The desired assertion follows by combining the last bound and Lemma \ref{dd:le:cou2} and \ref{dd:le:cou1} below.
\end{proof}

\begin{lem}\label{dd:le:cou2} Under assumptions of Proposition \ref{dd:co:dd:as1}.  Suppose that $\gB:= 2\sum_{k=0}^\infty(k+1)\beta_k<\infty$
 and set  $\Psi(x):=\sum_{m\geq 1}^\infty x^{1/2} m^{1/2}\exp(-
m^{1/2}/(48x^{1/2}))<\infty$, for any $x>0$, and $K_n:=\gauss{ 4  \maxnormsup^2 \sqrt{\oDi} /\gamma_f}$   then
 there exists a numerical constant $C>0$ such that 
 for any $\mu_n\geq \{3+ 8 \sum_{k=K_n+1}^{q-1}\beta(X_0,X_k)\}$
 holds
\begin{multline*}
\sup_{\So\in\Socwr}\Ex\vectp{\max_{\oDi\leq \Di\leq n}\set{\sup_{t\in\Bz_\Di}|\cou{\overline{\nu_t^e}}|^2-
  6 \Di n^{-1} \maxnormsup^2 \mu_n }}\leq C n^{-1} \maxnormsup^2\bigg\{\mu_n
        \Psi\bigg(\frac{\Sor\gA\gB}{\maxnormsup^2
          \mu_n^2}\bigg)\\\hfill+  n q^2\exp\left(-\frac{n^{1/2}}{q}\frac{\mu_n^{1/2}}{144}\right)\bigg\};\\
\sup_{\So\in\Socwr}\Ex\vectp{\max_{\oDi\leq \Di\leq n}\set{\sup_{t\in\Bz_\Di}|\cou{\overline{\nu_t^o}}|^2-
  6 \Di n^{-1} \maxnormsup^2 \mu_n }}\leq C n^{-1} \maxnormsup^2\bigg\{\mu_n
        \Psi\bigg(\frac{\Sor\gA\gB}{\maxnormsup^2
          \mu_n^2}\bigg)\\\hfill+  n q^2\exp\left(-\frac{n^{1/2}}{q}\frac{\mu_n^{1/2}}{144}\right)\bigg\}.
\end{multline*}
\end{lem}

\begin{proof}[\dr Proof of Lemma \ref{dd:le:cou2}]
We prove  the
  first assertion, the proof of the second follows exactly in the same
  way and, hence we omit the details. We shall emphasise that  $\cou{\overline{\nu_t^e}}=
p^{-1}\sum_{l=1}^p\vec{v_t}(E_l)$ where $(E_l)_{l=1}^p$ are iid.,
which we use below without further reference. 
 Keep in mind that $\vec{v_t}(x):=(1/q)\sum_{i=1}^qv_t(x_i)$ and set
  $\vec{\bas_j}(x):=(1/q)\sum_{i=1}^q\bas_j(x_i)$ for
  $x\in[0,1]^q$. In order to apply Talagrand's inequality we compute the constants $h$, $H$ and
  $v$. Consider first $h$
  where
\begin{equation}\label{dd:le:cou2:pr:e1}\sup_{t\in\cB_m}\normInf{\vec{\nu_t}}^2
=\sup_{y\in [0,1]^q}\sum_{j=1}^\Di|\tfrac{1}{q}\sum_{i=1}^q\bas_j(y_i)|^2
\leq  \maxnormsup^2 m=:h^2
\end{equation}
employing the assumption \ref{mo:no:as:ba:i}. 
 Consider next $H$. From
property \ref{dd:as:cou3}, follows that
\begin{equation*}
\Ex\sup_{t\in\cB_m}|\cou{\overline{\nu_t^e}}|^2=\sum_{j=1}^\Di \Var\{\tfrac{1}{p}\sum_{l=1}^p\vec{\bas_j}(\couE_l)\}=\tfrac{1}{p}\sum_{j=1}^\Di \Var\{\vec{\bas_j}(\couE_1)\}
\end{equation*}
and hence exploiting the definition of $\vec{\bas_j}$ and the property
\ref{dd:as:cou1}, we have 
\begin{equation}\label{dd:le:cou2:pr:e2}
\Ex\sup_{t\in\cB_m}|\cou{\overline{\nu_t^e}}|^2=\tfrac{1}{p}\sum_{j=1}^\Di \Var\{\vec{\bas_j}(\couE_1)\}=\tfrac{1}{p}\sum_{j=1}^\Di \Var\{\vec{\bas_j}(E_1)\}=\tfrac{1}{p}\sum_{j=1}^\Di \Var\{\tfrac{1}{q}\sum_{i=1}^q\bas_j(X_i)\}
\end{equation}
We employ next Lemma \ref{dd:le:est:var}, thereby under the assumptions
\ref{mo:no:as:ba:i} and \ref{dd:as:de:i} we have for all
$K\in\{0,\dotsc,q-1\}$ and for any $q\geq1$
\begin{equation*}
\sum_{j=1}^\Di \Var\{\tfrac{1}{q}\sum_{i=1}^q\bas_j(X_i)\}\leq    \frac{\Di}{q} \{\maxnormsup^2 + 2 [ \gamma K/\sqrt{\Di} + 2
\maxnormsup^2 \sum_{k=K+1}^{q-1}\beta(X_0,X_k)]\}.
\end{equation*}
Given $K_n=\gauss{ 4  \maxnormsup^2 \sqrt{\oDi} /\gamma}$
 we have  $\sum_{j=1}^\Di \Var\{\tfrac{1}{q}\sum_{i=1}^q\bas_j(X_i)\}\leq\frac{\Di}{q} \maxnormsup^2\{3/2 + 4 \sum_{k=K_n+1}^{q-1}\beta_k\}$, for all $m\geq \oDi$. 
Thereby, from \eqref{dd:le:cou2:pr:e2} follows for any 
$\mu_n\geq  \{3+ 8 \sum_{k=K_n+1}^{\infty}\beta_k\}$ that 
\begin{equation}\label{dd:le:cou2:pr:e3}
\Ex\sup_{t\in\cB_m}|\cou{\overline{\nu_t^e}}|^2\leq \frac{\Di}{n}
\maxnormsup^2  \{3+ 8 \sum_{k=K_n+1}^{\infty}\beta_k\} \leq \frac{\Di}{n}
\maxnormsup^2\mu_n=:H^2.
\end{equation}
Consider $v$. Keep in mind that $	\sup_{t\in\cB_m}\frac{1}{p}\sum_{i=1}^p
        \Var(\vec{\nu_t}(\couE_i))
=\sup_{t\in\cB_m}
\Var (\tfrac{1}{q}\sum_{i=1}^q v_t(X_i))$ due to  \ref{dd:as:cou1} and \ref{dd:as:cou3},
$\sup_{t\in\cB_m}\Ex|v_t(X_1)|^2\leq\Sor\gA$, and $\sup_{t\in\cB_m}\normInf{v_t}\leq
\Di^{1/2} \maxnormsup$ given in \eqref{prop:adaptivity:density:e3} and \eqref{prop:adaptivity:density:e1}, respectively. By applying
\eqref{dd:le:var:e3} and setting $\gB= 2\sum_{k=0}^\infty
(k+1)\beta_k$ we have
\begin{multline}\label{dd:le:cou2:pr:e4}
	\sup_{t\in\cB_m}\frac{1}{p}\sum_{i=1}^p
        \Var(\vec{\nu_t}(\couE_i))
        \leq \tfrac{4}{q}\sup_{t\in\cB_m}\{\Ex
|v_t(X_1)|^2\}^{1/2}\normInf{v_t}\{2\sum_{k=0}^\infty
(k+1)\beta_k\}^{1/2}\\\leq \tfrac{4}{q} (\Di\Sor\gA\gB)^{1/2}
\maxnormsup=:v.
\end{multline}
The assertion follows from 
 Lemma \ref{id:ka:l:talagrand} by using the quantities $h$, $H$ and
$v$ given in \eqref{dd:le:cou2:pr:e1},
\eqref{dd:le:cou2:pr:e3} and
\eqref{dd:le:cou2:pr:e4}, respectively, and by employing  the definition of $\Psi$, which completes the proof.
\end{proof}

\begin{lem}\label{dd:le:cou1} Under assumptions of Proposition \ref{dd:co:dd:as1}. We have  
\begin{equation*}
\Ex\vectp{\sup_{t\in\Bz_{n}}|{\overline{\nu_t^e}}-\cou{\overline{\nu_t^e}}|^2}\leq
4 \maxnormsup^2n\beta_{q+1},\quad \mbox{and,}\quad
\Ex\vectp{\sup_{t\in\Bz_{n}}|{\overline{\nu_t^o}}-\cou{\overline{\nu_t^o}}|^2}\leq4 \maxnormsup^2n\beta_{q+1}.
\end{equation*}
\end{lem}
\begin{proof}[\dr Proof of Lemma \ref{dd:le:cou1}] Since
  $\{E_l\}_{l=1}^p$ and   $\{\couE_l\}_{l=1}^p$ are identically
  distributed due to \ref{dd:as:cou1} we have $|{\overline{\nu_t^e}}-\cou{\overline{\nu_t^e}}|=|p^{-1}\sum_{l=1}^p
    \{\vec{v_t}(E_l)-\vec{v_t}(\couE_l)\}|\leq 2\normInf{\vec{\nu_t}}\1_{\{E_l\ne
      \couE_l\}}$ and hence, by using \ref{dd:as:cou2} it follows that 
    \begin{equation*}
\Ex\vectp{\sup_{t\in\cB_{n}}|{\overline{\nu_t^e}}-\cou{\overline{\nu_t^e}}|^2}\leq
4\sup_{t\in\cB_{n}}\normInf{\vec{\nu_t}}^2 p^{-1}\sum_{l=1}^p P(E_l\ne
      \couE_l)\leq 4\sup_{t\in\cB_{n}}\normInf{\vec{\nu_t}}^2\beta_{q+1}
    \end{equation*}
which together with \eqref{dd:le:cou2:pr:e1} shows the first assertion. The proof of
the second assertion is made exactly in the same way,  and hence we
omit the details, which completes the proof.
\end{proof}

\subsection{Appendix:  Proofs of Section \ref{s:dd:r}}\label{a:dd:r}

\begin{proof}[\dr Proof of Lemma \ref{dd:r:p:ub}]
Exploiting the assumption
\ref{mo:no:as:ba:i} and Lemma \ref{dd:le:var} we obtain,
  \begin{multline}\label{dd:r:le:var:e1}
\sum_{j=1}^\Di\Var(\tfrac{1}{n}\sum_{i=1}^n(\rNoL\rNo_i+\So(\rRe_i))\bas_j(\rRe_i))
\leq \frac{\rNoL^2 m}{n} + \frac{1}{n}
\inormV{\So}^2\inormV{\sum_{j=1}^\Di\bas_j^2}\{1+4\sum_{k=1}^{n-1}\beta(\rRe_0,\rRe_k)\}\\
\leq [\rNoL^2 +\inormV{\So}^2\maxnormsup^2\{1+4\sum_{k=1}^{n-1}\beta(\rRe_0,\rRe_k)\} ] \Di n^{-1}.
\end{multline}
Replacing \eqref{id:d:pr:var} by \eqref{dd:r:le:var:e1}, the assertion follows as in the proof of Proposition \ref{id:d:p:ub}.
\end{proof}

\begin{proof}[\dr Proof of Lemma \ref{dd:r:l:est:var}]
We start the proof with the observation that for any orthonormal
system $\{\bas_j\}_{j=1}^m$ we have $\HnormV{\sum_{j=1}^\Di\bas_j\otimes\bas_j}^2=\sum_{j=1}^\Di\sum_{l=1}^\Di|\skalarV{\bas_j,\bas_l}|^2
 = m$. Thereby, from \ref{dd:as:re:i} follows 
\begin{multline}\label{dd:r:l:est:var:pr:e1}
\bigg|\sum_{j=1}^\Di\Cov(\So(\rRe_0)\bas_j(\rRe_0),\So(\rRe_k)\bas_j(\rRe_k))\bigg|\\
\leq \HnormV{\sum_{j=1}^\Di\bas_j\otimes\bas_j}\HnormV{\So\otimes\So \{f_{\rRe_0,\rRe_k}-\1\otimes\1\}}
\leq \sqrt{m}\;\HnormV{\So}^2\gamma
\end{multline}
On the other hand side, keeping in mind \ref{mo:no:as:ba:i} there exists a function $b_k:\Rz\to[0,1]$ with  $\Ex b_k(\rRe_0)=\beta(\rRe_0,\rRe_k)$ due to Lemma 4.1 in \cite{Viennet1997} such that 
 \begin{multline*}\label{dd:r:l:est:var:pr:e2}
\bigg|\sum_{j=1}^\Di\Cov(\So(\rRe_0)\bas_j(\rRe_0),\So(\rRe_k)\bas_j(\rRe_k))\bigg|\leq 2
\Ex(b_k(\rRe_0)\{\So^2(\rRe_0)\sum_{j=1}^\Di \bas_j^2(\rRe_0)\})\\
\leq 2 m \inormV{\So}^2\maxnormsup^2 \beta(\rRe_0,\rRe_k)
 \end{multline*}
which together with \eqref{dd:r:l:est:var:pr:e1} implies for any $0\leq K \leq n-1$
\begin{multline*}
\sum_{k=1}^{n-1}(n+1-k)\sum_{j=1}^\Di\Cov(\So(\rRe_0)\bas_j(\rRe_0),\So(\rRe_k)\bas_j(\rRe_k))
\leq \sqrt{m}\inormV{\So}^2 \gamma  n K \\+ 2 m \inormV{\So}^2\maxnormsup^2 n
\sum_{k=K+1}^{n-1}\beta(\rRe_0,\rRe_k) = m n\, \inormV{\So}^2\{ \gamma K/\sqrt{m} + 2
\maxnormsup^2 \sum_{k=K+1}^{n-1}\beta(\rRe_0,\rRe_k)\}.
\end{multline*}
From the last bound and $\sum_{j=1}^\Di \sum_{i=1}^n \Var(\So(\rRe_i)\bas_j(\rRe_i))\leq n m \maxnormsup^2 \HnormV{\So}^2$ due to \ref{mo:no:as:ba:i}  follows the desired assertion.
\end{proof}

\begin{proof}[\dr Proof of Proposition \ref{dd:r:co:dd:as1}]
Recalling the notations given in the proof of Proposition \ref{id:r:p:co}, our proof starts with the observation that a combination of \eqref{id:r:dd:e1} and \eqref{id:r:p:co:pr:u} leads to 
\begin{equation}\label{tralalallal}
\Ex\vectp{\max_{\oDi\leq \Di\leq \DiMa}\{\HnormV{\hDiSo-\DiSo}^2- \tfrac{\pen}{6}\}}\\
 \leq{2}\Ex\vectp{\max_{\oDi\leq \Di\leq n}\{\sup_{t\in\Bz_\Di}|{\overline{\nu_t^b}}|^2- \tfrac{\pen}{12}\}}+2n^{-1} \rNoL^2\maxnormsup^2\Ex(\rNo^{6}).
\end{equation} 
In order to bound the first rhs. term we use a construction similar to that in the proof of Proposition \ref{dd:co:dd:as1}. Let $(\rRe_i)_{i\geq1}=(E_l,O_l)_{l\geq1}$ and
$(\cou{\rRe}_i)_{i\geq1}=(\couE_l,\couO_l)_{l\geq1}$ be random vectors  satisfying the coupling properties
\ref{dd:as:cou1}, \ref{dd:as:cou2} and \ref{dd:as:cou3}. Introduce exactly in the same manner $(\rNo^b_i)_{i\geq 1}=(\vec{\rNo}^{\;be}_l,\vec{\rNo}^{\;bo}_l)_{l\geq1}$. 
 If
we set $\vec{v}_t(x,y):=(1/q)\sum_{i=1}^qv_t(x_i,y_i)$, 
then for $n=2pq$ it follows
\begin{equation*}
 \overline{\nu_t^b} =
 \tfrac{1}{2}\big\{\tfrac{1}{p}\sum_{l=1}^p\{\vec{v}_t(\vec{\rNo}^{\;be}_l,E_l)-\Ex\vec{v}_t(\vec{\rNo}^{\;be}_l,E_l)
 \}+ \tfrac{1}{p}\sum_{l=1}^p\{\vec{v}_t(\vec{\rNo}^{\;bo}_l,O_l)-\Ex\vec{v}_t(\vec{\rNo}^{\;bo}_l,O_l)\}\big\}=:\tfrac{1}{2}\{\overline{\nu_t^{be}}+\overline{\nu_t^{bo}}\}. 
\end{equation*}
Considering  the random variables
$(\cou{\rRe}_i)_{i\geq1}$ rather than $(\rRe_i)_{i\geq1}$ we introduce in addition 
\begin{equation*}
\cou{ \overline{\nu_t^b}} =
 \tfrac{1}{2}\big\{\tfrac{1}{p}\sum_{l=1}^p\{\vec{v}_t(\vec{\rNo}^{\;be}_l,\cou{E}_l)-\Ex\vec{v}_t(\vec{\rNo}^{\;be}_l,\cou{E}_l)
 \} +\tfrac{1}{p}\sum_{l=1}^p\{\vec{v}_t(\vec{\rNo}^{\;bo}_l,\cou{O}_l)-\Ex\vec{v}_t(\vec{\rNo}^{\;bo}_l,\cou{O}_l)\}\big\}=:\tfrac{1}{2}\{\cou{\overline{\nu_t^{be}}}+\cou{\overline{\nu_t^{bo}}}\}.
\end{equation*}
As in the proof of Proposition \ref{dd:co:dd:as1}, it follows that
\begin{multline*}
\Ex\vectp{\max_{\oDi\leq \Di\leq \DiMa}\{\sup_{t\in\Bz_\Di}|{\overline{\nu_t^b}}|^2-  \tfrac{\pen}{12}\}}\leq
\Ex\vectp{\max_{\oDi\leq \Di\leq\DiMa}\{\sup_{t\in\Bz_\Di}|\cou{\overline{\nu_t^{be}}}|^2-
  \tfrac{\pen}{24}\}}+\Ex\vectp{\sup_{t\in\Bz_{\DiMa}}|\cou{\overline{\nu_t^{be}}}-{\overline{\nu_t^{be}}}|^2}\\+
\Ex\vectp{\max_{\oDi\leq \Di\leq \DiMa}\{\sup_{t\in\Bz_\Di}|\cou{\overline{\nu_t^{bo}}}|^2-
  \tfrac{\pen}{24}\}}+\Ex\vectp{\sup_{t\in\Bz_{\DiMa}}|\cou{\overline{\nu_t^{bo}}}-{\overline{\nu_t^{bo}}}|^2}.
\end{multline*} 
The desired assertion follows by combining \eqref{tralalallal}, the last bound, Lemma \ref{dd:le:cou1} and \ref{dd:r:le:cou2} .
\end{proof}

\begin{lem}\label{dd:r:le:cou2} Let the assumptions  \ref{mo:no:as:ba:i},  \ref{mo:no:as:ba:ii},
 \ref{dd:as:cou1}, \ref{dd:as:cou3}, and  \ref{dd:as:re:i} be
 satisfied. Suppose that $\gB:= 2\sum_{k=0}^\infty(k+1)\beta_k<\infty$.
Let $K_n:=\gauss{ \maxnormsup^2 \HnormV{\So}^2
  \sqrt{\oDi} /(\gamma \Sor^2\gA^2)}$  and $\mu_n\geq \{3+ 8
\sum_{k=K_n+1}^{\infty}\beta_k\}$. There exist a finite constant
$\zeta(\Sor\gA,\rNoL,\maxnormsup,\gB)$ depending on the quantities
$\Sor\gA$, $\rNoL$, $\maxnormsup$ and $\gB$ only and a numerical constant $C>0$ such that 
 for any 
 holds
\begin{multline*}
 \sup_{\So\in\Socwr}\Ex\vectp{\max_{\oDi\leq \Di\leq n}\sup_{t\in\Bz_\Di}|\cou{\overline{\nu_t^{be}}}|^2-
  6\tfrac{\Di}{n}\sigma_Y^2\maxnormsup^2  \mu_n }\\\hfill\leq C n^{-1} \maxnormsup^2(\rNoL+\Sor\gA)^2\bigg\{\zeta(\Sor\gA,\rNoL,\maxnormsup,\gB)+  n^{3/2} q^2\exp\left(-\frac{n^{1/4}}{q}\frac{1}{576(1+\Sor\gA/\rNoL)}\right)\bigg\};\\
\hspace{-7.2cm}\sup_{\So\in\Socwr}\Ex\vectp{\max_{\oDi\leq \Di\leq
    n}\sup_{t\in\Bz_\Di}|\cou{\overline{\nu_t^{bo}}}|^2-6\tfrac{\Di}{n}\sigma_Y^2\maxnormsup^2
  \mu_n }\\\hfill\leq C n^{-1} \maxnormsup^2(\rNoL+\Sor\gA)^2\bigg\{\zeta(\Sor\gA,\rNoL,\maxnormsup,\gB)+  n^{3/2} q^2\exp\left(-\frac{n^{1/4}}{q}\frac{1}{576(1+\Sor\gA/\rNoL)}\right)\bigg\}.
\end{multline*}
\end{lem}

\begin{proof}[\dr Proof of Lemma \ref{dd:r:le:cou2}] We prove  the
  first assertion, the proof of the second follows exactly in the same
  way and, hence we omit the details. 
In order to apply  Talagrand's inequality given in Lemma
  \ref{id:ka:l:talagrand} we need to compute the constants $h$, $H$ and
  $v$ which verify the three required inequalities. Keep in mind that
$\cou{\overline{\nu_t^{be}}}=\tfrac{1}{p}\sum_{l=1}^p\vec{v}_t(\vec{\rNo}^{\;be}_l,\couE_l)-\Ex\vec{v}_t(\vec{\rNo}^{\;be}_l,\couE_l)$
with
$\vec{v}_t(\vec{\rNo}^{\;be}_l,\couE_l)=\sum_{j=1}^\Di\fou{t}_j\vec{\psi_j}(\vec{\rNo}^{\;be}_l,\couE_l)$,
 $\vec{\psi_j}(\vec{\rNo}^{\;be}_l,\couE_l)=(1/q)\sum_{i\in\cI_l^e}\psi_j(\rNo^b_i,\cou{U}_i)$
and 
$\psi_j(\rNo^b_i,\cou{U}_i)=(\rNoL \rNo^b_i+\So(\cou{U}_i))\bas_j(\cou{U}_i)$,
where $|\vec{\rNo}^{\;be}_l|_\infty=\max_{i\in\cI_l^e}|\rNo^b_i|\leq2n^{1/4}$ and
  $\couE_l\in[0,1]^q$. Consider first $h$. As in
  \eqref{id:r:p:co:pr:h}, the assumption \ref{mo:no:as:ba:i} implies
\begin{equation}\label{dd:r:le:cou2:pr:h}\sup_{t\in\cB_m}\normInf{\vec{\nu_t}}^2
=\sum_{j=1}^\Di\inormV{\vec{\psi_j}}^2\leq\sum_{j=1}^m\normInf{\psi_j^2} \leq  \maxnormsup^2 m (2\rNoL n^{1/4} +\inormV{\So})^2=:h^2.
\end{equation}
 Consider next $H$. Exploiting successfully  
property \ref{dd:as:cou3}, 
the definition of $\vec{\psi_j}$ and the property
\ref{dd:as:cou1} together with the independence  within $\{\rNo_i\}$ and between  $\{\rNo_i\}$ and $\{\rRe_i\}$ we have 
\begin{equation}\label{dd:r:le:cou2:pr:e1}
\Ex\sup_{t\in\cB_m}|\cou{\overline{\nu_t^{be}}}|^2
\leq\frac{2\Di\rNoL^2\maxnormsup^2}{n}+\tfrac{1}{p}\sum_{j=1}^\Di\Var\big(\tfrac{1}{q}\sum_{i=1}^q\So(\rRe_i)\bas_j(\rRe_i)\big).
\end{equation}
Given  $K_n=\gauss{ 4  \maxnormsup^2 \HnormV{\So}^2
  \sqrt{\oDi} /(\gamma \Sor^2\gA^2)}$, Lemma \ref{dd:r:l:est:var}, assumptions
\ref{mo:no:as:ba:i} and \ref{dd:as:re:i} imply together  for all $m\geq \oDi$ that
\begin{equation*}
\sum_{j=1}^\Di \Var\big(\tfrac{1}{q}\sum_{i=1}^q\So(\rRe_i)\bas_j(\rRe_i)\big)\leq  \frac{\Di}{q} \maxnormsup^2\HnormV{\So}^2\{3/2 + 4 \sum_{k=K_n+1}^{q-1}\beta(\rRe_0,\rRe_k)]\}.
\end{equation*}
Thereby, from \eqref{dd:r:le:cou2:pr:e1} follows for any $\mu_n\geq \{3+ 8 \sum_{k=K_n+1}^{\infty}\beta_k\}$ that
\begin{equation}\label{dd:r:le:cou2:pr:H}
\Ex\sup_{t\in\cB_m}|\cou{\overline{\nu_t^e}}|^2\leq \frac{2\Di}{n}\rNoL^2\maxnormsup^2+\frac{\Di}{n}
\maxnormsup^2 \HnormV{\So}^2 \mu_n \leq \frac{\Di}{n}
\maxnormsup^2\sigma_Y^2\mu_n=:H^2.
\end{equation}
Consider finally $v$. Employing successively \ref{dd:as:cou3}, \ref{dd:as:cou1} and  \eqref{dd:le:var:e3} we have
\begin{multline}\label{dd:r:le:cou2:pr:e2}
	\sup_{t\in\cB_m}\tfrac{1}{p}\sum_{l=1}^p
        \Var(\vec{\nu_t}(\vec{\rNo}^{\;be}_l,\couE_l))
\leq \frac{\rNoL^2}{q}+\sup_{t\in\cB_m}\Var(\tfrac{1}{q}\sum_{i=1}^q\So(\rRe_i)\sum_{j=1}^m\fou{t}_j\bas_j(\rRe_i))\\
\leq  \frac{\rNoL^2}{q}+ \tfrac{4}{q}\sup_{t\in\cB_m}\{\Ex
|\So(\rRe_i)\sum_{j=1}^m\fou{t}_j\bas_j(\rRe_i)|^2\}^{1/2}\normInf{\So\sum_{j=1}^m\fou{t}_j\bas}\{2\sum_{k=0}^\infty
(k+1)\beta_k\}^{1/2}.
\end{multline}
Since 
$\sup_{t\in\cB_m}\Ex|\So(\rRe_i)\sum_{j=1}^m\fou{t}_j\bas_j(\rRe_i)|^2\leq\normInf{\So}^2$, $\sup_{t\in\cB_m}\normInf{\So\sum_{j=1}^m\fou{t}_j\bas_j}^2\leq
\Di \maxnormsup ^2\inormV{\So}^2$ and $\gB= 2\sum_{k=0}^\infty
(k+1)\beta_k$ it follows that 
\begin{equation}\label{dd:r:le:cou2:pr:v}
	\sup_{t\in\cB_m}\frac{1}{p}\sum_{l=1}^p
        \Var(\vec{\nu_t}(\vec{\rNo}^{\;be}_l,\couE_l))\\
\leq \frac{\Di^{1/2}\maxnormsup}{q}(\rNoL^2+4\normInf{\So}^2\gB^{1/2})=:v.
\end{equation}
The assertion follows from 
 Lemma \ref{id:ka:l:talagrand} by using the quantities $h$, $H$ and
$v$ given in \eqref{dd:r:le:cou2:pr:h},
\eqref{dd:r:le:cou2:pr:H} and
\eqref{dd:r:le:cou2:pr:v}, respectively, and by employing $\mu_n\geq 3/2$, $(\rNoL +\inormV{\So})^2/\sigma_Y^2\leq
2(1+\inormV{\So}/\rNoL)^2$, and  $\normInf{\So}\leq \Sor\gA$ for all $\So\in\Socwr$,  which completes the proof.
\end{proof}

\begin{proof}[\dr Proof of Lemma \ref{dd:r:l:re}]
Since $\Ex Y_1^2=\sigma_Y^2$ using successively the Tchebysheff inequality, the inequality \eqref{dd:le:var:e3}, the Cauchy-Schwarz inequality and Lemma \ref{dd:le:b} we get
\begin{equation*}
\proba{\left|n^{-1}\sum_{i=1}^n\left(\frac{Y_i^2}{\sigma_Y^2}-1\right)\right|\geq\frac{1}{2}}
\leq 16n^{-1}(\Ex Y_1^4/\sigma_Y^4)^{1/2}(2\sum_{k=0}^\infty(k+1)\beta_k)^{1/2}
\end{equation*}
which implies with $\Ex(Y_1^4/\sigma_Y^4)\leq 8 \frac{\sigma^4\Ex\epsilon^4+\normInf{\So}^4}{(\sigma^2+\normV{\So}^2)^2}
\leq32\{\frac{\sigma(\Ex\epsilon^4)^{1/4}+\normInf{\So}}{\sigma+\normV{\So}}\}^{4}
$ the desired assertion.
\end{proof}

%%% Local Variables: 
%%% mode: latex
%%% TeX-master: "_0DP_NPE_dep"
%%% End: 
  
% --------------------------------------------------------------------
% <<BibFile>>
% --------------------------------------------------------------------
\bibliography{./DP_NPE_dep} 
\end{document}